\documentclass[12pt]{amsart}
\usepackage{amssymb}
\usepackage{hyperref}
\usepackage[all]{xy}

\newtheorem{theorem}{Theorem}[section]
\newtheorem{definition}[theorem]{Definition}
\newtheorem{proposition}[theorem]{Proposition}

\begin{document}

\title{The algebraic structure of quantum partial isometries}

\author{Teodor Banica}
\address{T.B.: Department of Mathematics, Cergy-Pontoise University, 95000 Cergy-Pontoise, France. {\tt teodor.banica@u-cergy.fr}}


\begin{abstract}
The partial isometries of $\mathbb R^N,\mathbb C^N$ form compact semigroups $\widetilde{O}_N,\widetilde{U}_N$. We discuss here the liberation question for these semigroups, and for their discrete versions $\widetilde{H}_N,\widetilde{K}_N$. Our main results concern the construction of half-liberations $\widetilde{H}_N^\times,\widetilde{K}_N^\times,\widetilde{O}_N^\times,\widetilde{U}_N^\times$ and of liberations $\widetilde{H}_N^+,\widetilde{K}_N^+,\widetilde{O}_N^+,\widetilde{U}_N^+$. We include a detailed algebraic and probabilistic study of all these objects, justifying our ``half-liberation'' and ``liberation'' claims.
\end{abstract}

\maketitle

\section*{Introduction}

A remarkable discovery, due to Wang \cite{wa1}, is that the orthogonal and unitary groups $O_N,U_N$ have free analogues $O_N^+,U_N^+$. More precisely, inspired by some previous work by Brown \cite{bro} and McClanahan \cite{mcc}, Wang considered the following two algebras:
\begin{eqnarray*}
C(O_N^+)&=&C^*\left((u_{ij})_{i,j=1,\ldots,N}\Big|u=\bar{u},u^{-1}=u^t\right)\\
C(U_N^+)&=&C^*\left((u_{ij})_{i,j=1,\ldots,N}\Big|u^{-1}=u^*,\bar{u}^{-1}=u^t\right)
\end{eqnarray*}

These two algebras have a comultiplication, counit, and antipode, defined respectively by $\Delta(u_{ij})=\sum_ku_{ik}\otimes u_{kj}$, $\varepsilon(u_{ij})=\delta_{ij}$, $S(u_{ij})=u_{ji}^*$. The general axioms found by Woronowicz in \cite{wo1}, \cite{wo2} are satisfied, and the underlying compact quantum spaces $O_N^+,U_N^+$ are therefore compact quantum groups, called free analogues of $O_N,U_N$.

The present paper is mainly concerned with the liberation question for $\widetilde{O}_N,\widetilde{U}_N$, the semigroups of partial isometries of $\mathbb R^N,\mathbb C^N$. We have several motivations for doing this work, coming from linear algebra, noncommutative geometry and free probability.

Our main sources of inspiration will be on one hand the recent advances by Bichon and Dubois-Violette \cite{bdu} and by Bhowmick, D'Andrea and Dabrowski \cite{bdd} on the half-liberation operation, and on the other hand our recent paper with Skalski \cite{bsk}, where a free analogue of the semigroup $\widetilde{S}_N$ of quantum partial permutations was constructed.

We will discuss in fact the liberation problem for several discrete and compact semigroups. We will mostly focus on six ``fundamental'' examples, as follows:
$$\begin{matrix}
\widetilde{B}_N&\subset&\widetilde{O}_N&\subset&\widetilde{U}_N\\
\\
\cup&&\cup&&\cup\\
\\
\widetilde{S}_N&\subset&\widetilde{H}_N&\subset&\widetilde{K}_N
\end{matrix}$$

The choice of these particular semigroups comes from the ``easy quantum group'' philosophy, emerging from \cite{bb+}, \cite{bbc}, \cite{bc1}, \cite{bc2}, and that we axiomatized with Speicher in \cite{bsp}. 

We will construct here liberations and half-liberations of the above 6 semigroups, by performing a case-by-case analysis. Our plan of study will be as follows:
\begin{enumerate}
\item We will first investigate the discrete case, concerning $\widetilde{S}_N,\widetilde{H}_N,\widetilde{K}_N$. Our study here will be based on \cite{bsk}, with probabilistic ingredients from \cite{bb+}, \cite{bbc}, \cite{bc2}, \cite{bv1}. 

\item We will investigate then the continuous case, concerning $\widetilde{B}_N,\widetilde{O}_N,\widetilde{U}_N$. Here we will use some extra algebraic ingredients, coming from \cite{bdd}, \cite{bdu}.
\end{enumerate}

Our results will be mostly theoretical, providing a framework for the study of various semigroups of quantum partial isometries $G\subset\widetilde{U}_N^+$. Regarding now the motivations and potential applications, as already mentioned, these come from linear algebra, noncommutative geometry, and free probability. More precisely, these are as follows:
\begin{enumerate}
\item As explained in \cite{bsk}, certain semigroups $G\subset\widetilde{S}_N^+$ coming from matrix models are of interest in connection with linear algebra questions. The problematics here is waiting to be extended, and investigated, using methods from \cite{bfs}, \cite{ss2}, \cite{wa3}.

\item An interesting question regards the general understanding of the noncommutative homogeneous spaces \cite{boc}, \cite{pod}. There is a certain similarity between the present work and the one in \cite{bss}, waiting to be understood, and axiomatized.

\item Finally, the quantum groups $S_N^+,O_N^+,U_N^+$ have been successfully used in connection with several free probability questions \cite{cur}, \cite{csp}, \cite{fwe}, \cite{ksp}. The semigroup extension of these results is an open problem, that we would like to raise here.
\end{enumerate}

The paper is organized as follows: in 1-2 we discuss the partial permutation case, in 3-4 we discuss various discrete and continuous extensions, and in 5-6 we discuss the half-liberation operation, first in the orthogonal case, and then in the unitary case.

\medskip

\noindent {\bf Acknowledgements.} I would like to thank Adam Skalski for useful discussions. This work was partly supported by the NCN grant 2012/06/M/ST1/00169.

\section{Partial permutations}

We are interested in the liberation problem for various semigroups of partial isometries. The simplest such semigroup is the one which permutes the coordinate axes of $\mathbb R^N$.

This semigroup is best introduced combinatorially, as follows:

\begin{definition}
$\widetilde{S}_N$ is the semigroup of partial permutations of $\{1\,\ldots,N\}$,
$$\widetilde{S}_N=\{\sigma:X\simeq Y|X,Y\subset\{1,\ldots,N\}\}$$
with the usual composition operation, $\sigma'\sigma:\sigma^{-1}(X'\cap Y)\to\sigma'(X'\cap Y)$.
\end{definition}

Observe that $\widetilde{S}_N$ is not simplifiable, because the null permutation $\emptyset\in\widetilde{S}_N$, having the empty set as domain/range, satisfies $\emptyset\sigma=\sigma\emptyset=\emptyset$, for any $\sigma\in\widetilde{S}_N$. Observe also that $\widetilde{S}_N$ has a ``subinverse'' map, sending $\sigma:X\to Y$ to its usual inverse $\sigma^{-1}:Y\simeq X$.

A first result on $\widetilde{S}_N$, already pointed out in \cite{bsk}, is as follows:

\begin{proposition}
The number of partial permutations is given by
$$|\widetilde{S}_N|=\sum_{k=0}^Nk!\binom{N}{k}^2$$
that is, $1,2,7,34,209,\ldots$, and with $N\to\infty$ we have $|\widetilde{S}_N|\simeq N!\sqrt{\frac{\exp(4\sqrt{N}-1)}{4\pi\sqrt{N}}}$.
\end{proposition}

Another result, which is trivial, but quite fundamental, is as follows:

\begin{proposition}
We have a semigroup embedding $u:\widetilde{S}_N\subset M_N(0,1)$, defined by 
$$u_{ij}(\sigma)=
\begin{cases}
1&{\rm if}\ \sigma(j)=i\\
0&{\rm otherwise}
\end{cases}$$
whose image are the matrices having at most one nonzero entry, on each row and column.
\end{proposition}

Finally, a third basic result about $\widetilde{S}_N$ is as follows:

\begin{proposition}
We have an embedding $\widetilde{S}_N\subset S_{2N}$, mapping $\sigma:X\simeq Y$ to
$$\sigma'(i)=
\begin{cases}
\sigma(i)&\ if\ i\in X\\
N+r&\ if\ i=x_r\\
y_r&\ if\ i=N+r\\
i&\ if\ i>N+L 
\end{cases}$$
where $X^c=\{x_1,\ldots,x_L\}$ and $Y^c=\{y_1,\ldots,y_L\}$, with $x_1<\ldots<x_L$ and $y_1<\ldots<y_L$.
\end{proposition}

Let us discuss now some probabilistic aspects. We denote by $\kappa:\widetilde{S}_N\to\mathbb N$ the cardinality of the domain/range, and by $\chi:\widetilde{S}_N\to\mathbb N$ be the number of fixed points. Observe that we have $\kappa=\sum_{ij}u_{ij}$ and $\chi=\sum_iu_{ii}$. These quantities are in fact of similar nature:

\begin{proposition}
The embedding $\widetilde{S}_N\subset S_{2N}$ in Proposition 1.4 makes correspond the variables $\chi,\kappa:\widetilde{S}_N\to\mathbb N$ to the variables $\chi_{left},\chi_{right}:S_{2N}\to\mathbb N$, counting the fixed points among $\{1,\ldots,N\}$ and among $\{N+1,\ldots,2N\}$, respectively.
\end{proposition}

\begin{proof}
By using the formula of $\sigma\to\sigma'$ from Proposition 1.4, we obtain:
\begin{eqnarray*}
\chi_{left}(\sigma')&=&\#\{i\leq N|\sigma'(i)=i\}=\{i\in X|\sigma(i)=i\}=\chi(\sigma)\\
\chi_{right}(\sigma')&=&\#\{i>N|\sigma'(i)=i\}=\{i>N+L\}=N-2=\kappa(\sigma)
\end{eqnarray*}

Thus we have obtained the formulae in the statement, and we are done.
\end{proof}

More generally, given a number $l\leq N$, we denote by $\chi_l:\widetilde{S}_N\to\mathbb N$ the number of fixed points among $\{1,\ldots,l\}$. Observe that $\widetilde{S}_N\subset S_{2N}$ maps in fact $\chi_l\to\chi_l$, for any $l$.

Generally speaking, we are interested in the joint law of $(\chi_l,\kappa)$. There are many interesting questions here, and as a main result on this subject, we have:

\begin{theorem}
The measures $\mu_k^l=law(\chi_l|\kappa=k)$ are given by
$$\mu_k^l=\sum_{q\geq0}\binom{k}{q}\binom{l}{q}\binom{N}{q}^{-2}\frac{(\delta_1-\delta_0)^{*q}}{q!}$$
and become Poisson ($st)$ in the $k=sN,l=tN,N\to\infty$ limit.
\end{theorem}

\begin{proof}
Observe first that at $k=l=N$ this corresponds to the well-known fact that the number of fixed points $\chi:S_N\to\mathbb N$ becomes Poisson (1), in the $N\to\infty$ limit. 

More generally, at $k=N$ this corresponds to the fact that the truncated character $\chi_l:S_N\to\mathbb N$ becomes Poisson ($t$), in the $l=tN\to\infty$ limit. See \cite{bc2}.

In general, we can use the same method, namely the inclusion-exclusion principle. If we set $\widetilde{S}_N^{(k)}=\{\sigma\in\widetilde{S}_N|\kappa(\sigma)=k\}$, the formula is as follows:
\begin{eqnarray*}
P(\chi_l=p|\kappa=k)
&=&\frac{1}{|\widetilde{S}_N^{(k)}|}\binom{l}{p}\#\left\{\sigma\in\widetilde{S}_{N-p}^{(k-p)}\Big|\sigma(i)\neq i,\forall i\leq l-p\right\}\\
&=&\frac{1}{|\widetilde{S}_N^{(k)}|}\binom{l}{p}\sum_{r\geq0}(-1)^r\binom{l-p}{r}\left|\widetilde{S}_{N-p-r}^{(k-p-r)}\right|
\end{eqnarray*}

Here the index $r$, which counts the fixed points among $\{1,\ldots,l-p\}$, runs a priori up to $\min(k,l)-p$. However, since the binomial coefficient or the cardinality of the set on the right vanishes by definition at $r>\min(k,l)-p$, we can sum over $r\geq0$.

Now by using $|\widetilde{S}_N^{(k)}|=k!\binom{N}{k}^2$, and then by cancelling various factorials, and grouping back into binomial coeffiecients, we obtain the following formula:
\begin{eqnarray*}
P(\chi_l=p|\kappa=k)
&=&\frac{1}{k!\binom{N}{k}^2}\binom{l}{p}\sum_{r\geq0}(-1)^r\binom{l-p}{r}(k-p-r)!\binom{N-p-r}{k-p-r}^2\\
&=&\sum_{r\geq0}\frac{(-1)^r}{p!r!}\binom{k}{p+r}\binom{l}{p+r}\binom{N}{p+r}^{-2}
\end{eqnarray*}

We can now compute the measure itself. With $p=q-r$, we obtain:
\begin{eqnarray*}
law(\chi_l|\kappa=k)
&=&\sum_{p\geq0}\sum_{r\geq0}\frac{(-1)^r}{p!r!}\binom{k}{p+r}\binom{l}{p+r}\binom{N}{p+r}^{-2}\,\delta_p\\
&=&\sum_{q\geq0}\sum_{r\geq0}\frac{(-1)^r}{(q-r)!r!}\binom{k}{q}\binom{l}{q}\binom{N}{q}^{-2}\,\delta_{q-r}\\
&=&\sum_{q\geq0}\binom{k}{q}\binom{l}{q}\binom{N}{q}^{-2}\cdot\frac{1}{q!}\sum_{r\geq0}(-1)^r\binom{q}{r}\delta_{q-r}
\end{eqnarray*}

The sum at right being $(\delta_1-\delta_0)^{*q}$, this gives the formula in the statement.

Regarding now the asymptotics, in the regime $k=sN,l=tN,N\to\infty$ from the statement, the coefficient of $(\delta_1-\delta_0)^{*q}/q!$ in the formula of $\mu_k^l$ is:
$$c_q=\binom{k}{q}\binom{l}{q}\binom{N}{q}^{-2}=\frac{\binom{k}{q}}{\binom{N}{q}}\cdot\frac{\binom{l}{q}}{\binom{N}{q}}\simeq\left(\frac{k}{N}\right)^q\left(\frac{l}{N}\right)^q=(st)^q$$

We deduce that the Fourier transform of $\mu_k^l$ is given by:
$$\mathcal F(\mu_k^l)(y)\simeq\sum_{q\geq0}(st)^q\frac{(e^y-1)^q}{q!}=e^{st(e^y-1)}$$

But this is the Fourier transform of Poisson ($st$), and we are done.
\end{proof}

Observe that the formula in Theorem 1.6 shows that we have $\mu_k^l=\mu_l^k$. This is an interesting equality, which seems to be quite unobvious to prove, with bare hands.

\section{Quantum permutations}

In this section we review some material from \cite{bsk}, where $\widetilde{S}_N^+$ was constructed, and we confirm that $\widetilde{S}_N\to\widetilde{S}_N^+$ is indeed a liberation, with a free version of Theorem 1.6. 

We use the formalism of compact matrix quantum groups, developed by Woronowicz in \cite{wo1}, \cite{wo2}. For a detailed presentation here, see  \cite{ntu}. We assume in addition that the square of the antipode is the identity, $S^2=id$. In short, we use the ``minimal'' formalism covering the compact Lie groups, and the duals of finitely generated discrete groups.

We recall that a square matrix $u=(u_{ij})$ is called ``magic'' if its entries are projections ($p=p^2=p^*$), which sum up to 1 on each row and column. The basic example is provided by the matrix coordinates $u_{ij}:S_N\subset O_N\to\mathbb R$, given by $u_{ij}(\sigma)=\delta_{i\sigma(j)}$.

The following key definition is due to Wang \cite{wa2}:

\begin{definition}
$C(S_N^+)$ is the universal $C^*$-algebra generated by the entries of a $N\times N$ magic matrix $u$, with comultiplication $\Delta(u_{ij})=\sum_ku_{ik}\otimes u_{kj}$, counit $\varepsilon(u_{ij})=\delta_{ij}$ and antipode $S(u_{ij})=u_{ji}$.
\end{definition}

This algebra satisfies the axioms in \cite{wo1}, \cite{wo2}, so the underlying noncommutative space $S_N^+$ is a compact quantum group, called quantum permutation group. The canonical embedding $S_N\subset S_N^+$ is an isomorphism at $N=1,2,3$, but not at $N\geq4$. See \cite{wa2}.

Let us go back now to the embedding $u:\widetilde{S}_N\subset M_N(0,1)$ in Proposition 1.3. Due to the formula $u_{ij}(\sigma)=\delta_{i\sigma(j)}$, the matrix $u=(u_{ij})$ is ``submagic'', in the sense that its entries are projections, which are pairwise orthogonal on each row and column. 

This suggests the following definition, given in \cite{bsk}:

\begin{definition}
$C(\widetilde{S}_N^+)$ is the universal $C^*$-algebra generated by the entries of a $N\times N$ submagic matrix $u$, with comultiplication $\Delta(u_{ij})=\sum_ku_{ik}\otimes u_{kj}$ and counit $\varepsilon(u_{ij})=\delta_{ij}$. 
\end{definition}

The bialgebra structure of $C(\widetilde{S}_N^+)$ tells us that the underlying noncommutative space $\widetilde{S}_N^+$ is a compact quantum semigroup. This semigroup is of quite special type, because $C(\widetilde{S}_N^+)$ has as well a ``subantipode'' map, defined by $S(u_{ij})=u_{ji}$. See \cite{bsk}.

Observe that $\Delta,\varepsilon,S$ restrict to $C(\widetilde{S}_N)$, and correspond there, via Gelfand duality, to the usual multiplication, unit element, and subinversion map of $\widetilde{S}_N$.

The basic properties of $\widetilde{S}_N^+$ can be summarized as follows:

\begin{proposition}
We have maps as follows
$$\begin{matrix}
C(\widetilde{S}_N^+)&\to&C(S_N^+)\\
\\
\downarrow&&\downarrow\\
\\
C(\widetilde{S}_N)&\to&C(S_N)
\end{matrix}
\quad \quad \quad:\quad \quad\quad
\begin{matrix}
\widetilde{S}_N^+&\supset&S_N^+\\
\\
\cup&&\cup\\
\\
\widetilde{S}_N&\supset&S_N
\end{matrix}$$
with the bialgebras at left corresponding to the quantum semigroups at right.
\end{proposition}

\begin{proof}
This is clear from the above discussion, and from the well-known fact that projections which sum up to $1$ are pairwise orthogonal. See \cite{bsk}. 
\end{proof}

As a first example, we have $\widetilde{S}_1^+=\widetilde{S}_1$. At $N=2$ now, recall that the algebra generated by two free projections $p,q$ is isomorphic to the group algebra of $D_\infty=\mathbb Z_2*\mathbb Z_2$. We denote by $\varepsilon:C^*(D_\infty)\to\mathbb C1$ the counit, $\varepsilon(1)=1$ and $\varepsilon(\ldots pqpq\ldots)=0$. We have:

\begin{proposition}
We have $C(\widetilde{S}_2^+)\simeq\{(x,y)\in C^*(D_\infty)\oplus C^*(D_\infty)|\varepsilon(x)=\varepsilon(y)\}$, via
$$u=\begin{pmatrix}p\oplus 0&0\oplus r\\0\oplus s&q\oplus 0\end{pmatrix}$$ where $p,q$ and $r,s$ are the standard generators of the two copies of $C^*(D_\infty)$.
\end{proposition}

\begin{proof}
Consider an arbitrary $2\times 2$ matrix formed by projections:
$$u=\begin{pmatrix}P&R\\S&Q\end{pmatrix}$$

This matrix is submagic when $PR=PS=QR=QS=0$, which means that the non-unital algebras $X=<P,Q>$ and $Y=<R,S>$ must commute, and must satisfy $xy=0$, for any $x\in X,y\in Y$. Thus $C(\widetilde{S}_2^+)\simeq\mathbb C1\oplus Z\oplus Z$, where $Z$ is the universal non-unital algebra generated by two projections. Since $C^*(D_\infty)=\mathbb C1\oplus Z$, we obtain an isomorphism $C(\widetilde{S}_2^+)\simeq\{(\lambda+a,\lambda+b)|\lambda\in\mathbb C, a,b\in Z\}$, as in the statement. See \cite{bsk}.
\end{proof}

The above result shows that the structure of $\widetilde{S}_2^+$ is quite non-trivial. Some further interpretations of Proposition 2.4, based on the work in \cite{ss1}, can be found in \cite{bsk}.

Let us now review the results in section 1. Proposition 1.2 has no free analogue, because $\widetilde{S}_N^+$ is infinite. Regarding Proposition 1.4 and Proposition 1.5, we first have:

\begin{proposition}
The following two elements of $C(\widetilde{S}_N^+)$ are self-adjoint,
$$\chi=\sum_iu_{ii},\qquad\kappa=\sum_{ij}u_{ij}$$
satisfy $0\leq\chi,\kappa\leq N$, and coincide with the usual $\chi,\kappa$ on the quotient $C(\widetilde{S}_N)$.
\end{proposition}

\begin{proof}
All the assertions are clear from definitions, with the inequalities $0\leq\chi,\kappa\leq N$ being taken of course in an operator-theoretic sense.
\end{proof}

With this observation in hand, if we denote by $v=(v_{ij})$ the magic unitary for $S_{2N}$, the formulae in Proposition 1.5 tell us that the surjection $C(S_{2N})\to C(\widetilde{S}_N)$ maps:
\begin{eqnarray*}
v_{11}+\ldots+v_{NN}&\to&\chi\\
v_{N+1,N+1}+\ldots+v_{2N,2N}&\to&\kappa
\end{eqnarray*}

Regarding now a free analogue $C(S_{2N}^+)\to C(\widetilde{S}_N^+)$, while at $N=2$ one can probably use Proposition 2.4 above, at $N=3$ we believe that the answer is negative. See \cite{bsk}.

Let us look now at Theorem 1.6. Since $C(\widetilde{S}_N^+)$ has no integration functional, we cannot speak about the joint law of $(\chi,\kappa)$. Thus, we need an alternative approach to $\mu_k^l$. 

For this purpose, we use the following simple observation:

\begin{proposition}
Any partial permutation $\sigma:X\simeq Y$ can be factorized as
$$\xymatrix{X\ar[r]^{\sigma}\ar[d]_\gamma&Y\\\{1,\ldots,k\}\ar[r]_\beta&\{1,\ldots,k\}\ar[u]_\alpha}$$
with $\alpha,\beta,\gamma\in S_k$ being certain (non-unique) permutations, where $k=\kappa(\sigma)$.
\end{proposition}

\begin{proof}
Since we have $|X|=|Y|=k$, we can choose any two bijections $X\simeq\{1,\ldots,k\}$ and $\{1,\ldots,k\}\simeq Y$, and then complete them up to permutations $\gamma,\alpha\in S_N$. The remaining permutation $\beta\in S_k$ is then uniquely determined by the formula $\sigma=\alpha\beta\gamma$.
\end{proof}

We can now formulate an alternative definition for the measures $\mu_k^l$. We fix $k\leq N$, and we denote by $p,q,r$ the magic matrices for $S_N,S_k,S_N$. We have:

\begin{proposition}
Consider the map $\varphi:S_N\times S_k\times S_N\to\widetilde{S}_N$, sending $(\alpha,\beta,\gamma)$ to the partial permutation $\sigma:\gamma^{-1}\{1,\ldots,k\}\simeq\alpha\{1,\ldots,k\}$ given by $\sigma(\gamma^{-1}(t))=\alpha(\beta(t))$. 
\begin{enumerate}
\item The image of $\varphi$ is the set $\widetilde{S}_N^{(k)}=\{\sigma\in\widetilde{S}_N|\kappa(\sigma)=k\}$.

\item The transpose of $\varphi$ is given by $\varphi^*(u_{ij})=\sum_{s,t\leq k}p_{is}\otimes q_{st}\otimes r_{tj}$.

\item $\mu_k^l$ equals the law of the variable $\varphi^*(\chi_l)\in C(S_N\times S_k\times S_N)$.
\end{enumerate}
\end{proposition}

\begin{proof}
This is an elementary statement, whose proof goes as follows:

(1) Since $\alpha,\gamma\in S_N$, the domain and range of the associated element $\sigma\in\widetilde{S}_N$ have indeed cardinality $k$. The surjectivity follows from Proposition 2.6 above.

(2) For the element $\sigma\in\widetilde{S}_N$ in the statement, we have:
\begin{eqnarray*}
u_{ij}(\sigma)=1
&\iff&\sigma(j)=i\\
&\iff&\exists t\leq k,\gamma^{-1}(t)=j,\alpha(\beta(t))=i\\
&\iff&\exists s,t\leq k,\gamma^{-1}(t)=j,\beta(t)=s,\alpha(s)=i\\
&\iff&\exists s,t\leq k,r_{tj}(\gamma)=1,q_{st}(\beta)=1,p_{is}(\alpha)=1\\
&\iff&\exists s,t\leq k,(p_{is}\otimes q_{st}\otimes r_{tj})(\alpha,\beta,\gamma)=1
\end{eqnarray*}

Now since the numbers $s,t\leq k$ are uniquely determined by $\alpha,\beta,\gamma,i,j$, if they exist, we conclude that we have the following formula:
$$u_{ij}(\sigma)=\sum_{s,t\leq k}(p_{is}\otimes q_{st}\otimes r_{tj})(\alpha,\beta,\gamma)$$

But this gives the formula in the statement, and we are done.

(3) This comes from the fact that the map $\varphi_k:S_N\times S_k\times S_N\to\widetilde{S}_N^{(k)}$ obtained by restricting the target of $\varphi$ commutes with the normalized (mass one) counting measures. At $k=N$ this follows from the well-known fact that given $(\alpha,\beta,\gamma)\in S_N\times S_N\times S_N$ random, the product $\alpha\beta\gamma\in S_N$ is random, and the general case is clear as well.
\end{proof}

The point now is that we can use the same trick, ``$\sigma=\alpha\beta\gamma$'', in the free case. The precise preliminary statement that we will need is as follows:

\begin{proposition}
Let $p,q,r$ be the magic matrices for $S_N^+,S_k^+,S_N^+$.
\begin{enumerate}
\item The matrix $U_{ij}=\sum_{s,t\leq k}p_{is}\otimes q_{st}\otimes r_{tj}$ is submagic.

\item We have a representation $\pi:C(\widetilde{S}_N^+)\to C(S_N^+\times S_k^+\times S_N^+)$, $\pi(u_{ij})=U_{ij}$.

\item $\pi$ factorizes through the algebra $C(\widetilde{S}_N^{+(k)})=C(\widetilde{S}_N^+)/<\kappa=k>$.

\item At $k=N$, this factorization $\pi_k$ commutes with the Haar functionals.
\end{enumerate}
\end{proposition}

\begin{proof}
Once again, this is an elementary statement, whose proof goes at follows:

(1) By using the fact that $p,q,r$ are magic, we obtain:
\begin{eqnarray*}
U_{ij}U_{il}
&=&\sum_{s,t\leq k}\sum_{v,w\leq k}p_{is}p_{iv}\otimes q_{st}q_{vw}\otimes r_{tj}r_{wl}\\
&=&\sum_{s,t\leq k}\sum_{w\leq k}p_{is}\otimes q_{st}q_{sw}\otimes r_{tj}r_{wl}\\
&=&\sum_{s,t\leq k}p_{is}\otimes q_{st}\otimes r_{tj}r_{tl}=\delta_{jl}U_{ij}
\end{eqnarray*}

The proof of $U_{ij}U_{lj}=\delta_{il}U_{ij}$ is similar, and we conclude that $U$ is submagic.

(2) This follows from (1), and from the definition of $C(\widetilde{S}_N^+)$.

(3) By using the fact that $p,q,r$ are magic, we obtain indeed:
$$\sum_{ij}U_{ij}=\sum_{ij}\sum_{s,t\leq k}p_{is}\otimes q_{st}\otimes r_{tj}=\sum_{s,t\leq k}1\otimes q_{st}\otimes 1=k$$

Thus the representation $\pi$ factorizes indeed through the algebra in the statement.

(4) This is a well-known analogue of the fact that ``the product of random permutations is a random permutation'', that we already used in the proof of Proposition 2.7 (3) above. Here is representation theory proof, using \cite{wo1}. With $P=Proj(Fix(u^{\otimes n}))$, we have:
\begin{eqnarray*}
\int_{S_N^+\times S_N^+\times S_N^+}U_{i_1j_1}\ldots U_{i_nj_n}
&=&\sum_{st}\int_{S_N^+}p_{i_1s_1}\ldots p_{i_ns_n}\int_{S_N^+}q_{s_1t_1}\ldots q_{s_nt_n}\int_{S_N^+}r_{t_1j_1}\ldots r_{t_nj_n}\\
&=&\sum_{st}P_{i_1\ldots i_n,s_1\ldots s_n}P_{s_1\ldots s_n,t_1\ldots t_n}P_{t_1\ldots t_n,j_1\ldots j_n}\\
&=&(P^3)_{i_1\ldots i_n,j_1\ldots j_n}=P_{i_1\ldots i_n,j_1\ldots j_n}\\
&=&\int_{S_N^+}u_{i_1j_1}\ldots u_{i_nj_n}
\end{eqnarray*}

Thus $\pi_N$ commutes indeed with the Haar functionals, and we are done.
\end{proof}

Observe that, since $\kappa$ is now continuous, $0\leq\kappa\leq N$, the algebras $C(\widetilde{S}_N^{+(k)})$ constructed in Proposition 2.8 don't sum any longer up to the algebra $C(\widetilde{S}_N^+)$ itself. Thus, in a certain sense, the above measures $\mu_k^l$ encode only a part of the ``probabilistic theory'' of $\widetilde{S}_N^+$.

We can however formulate a free analogue of Theorem 1.6, as follows:

\begin{theorem}
The measures $\mu_k^l=law(\pi_k(\chi_l))$, where $\pi_k$ is defined as
$$\pi_k:C(\widetilde{S}_N^+)\to C(S_N^+\times S_k^+\times S_N^+)\quad:\quad u_{ij}\to\sum_{s,t\leq k}p_{is}\otimes q_{st}\otimes r_{tj}$$
become free Poisson ($st$) in the $k=sN,l=tN,N\to\infty$ limit.
\end{theorem}

\begin{proof}
Observe first that at $k=l=N$ this corresponds to the fact that the law of the main character $\chi:S_N^+\to\mathbb N$ becomes free Poisson (1), in the $N\to\infty$ limit. Unlike in the classical case, the convergence here is stationary, starting from $N=4$. See \cite{bc2}.

More generally, at $k=N$ this corresponds to the fact that the truncated character $\chi_l:S_N^+\to\mathbb N$ becomes free Poisson ($t$), in the $l=tN\to\infty$ limit. See \cite{bc2}.

In general, we can use the same technique, namely the moment method, and the Weingarten formula. The variable that we are interested in, $\chi_k^l=\pi_k(\chi_l)$, is given by:
$$\chi_k^l=\sum_{i\leq l}\sum_{s,t\leq k}p_{is}\otimes q_{st}\otimes r_{ti}$$

By raising to the power $n$ and integrating, we obtain the following formula:
$$\int_{S_N^+\times S_k^+\times S_N^+}(\chi_k^l)^n=\sum_{i_a\leq l}\sum_{s_a,t_a\leq k}\int_{S_N^+}p_{i_1s_1}\ldots p_{i_ns_n}\int_{S_k^+}q_{s_1t_1}\ldots q_{s_nt_n}\int_{S_N^+}r_{t_1i_1}\ldots r_{t_ni_n}$$

By using now the Weingarten formula (\cite{bc2}), the above moment is:
\begin{eqnarray*}
c_n
&=&\sum_{ist}\sum_{\alpha\ldots\rho\in NC(n)}\delta_\alpha(i)\delta_\beta(s)W_{nN}(\alpha,\beta)\cdot\delta_\gamma(s)\delta_\delta(t)W_{nk}(\gamma,\delta)\cdot\delta_\varepsilon(t)\delta_\rho(i)W_{nN}(\varepsilon,\rho)\\
&=&\sum_{\alpha\ldots\rho\in NC(n)}W_{nN}(\alpha,\beta)W_{nk}(\gamma,\delta)W_{nN}(\varepsilon,\rho)\sum_{ist}\delta_\alpha(i)\delta_\beta(s)\delta_\gamma(s)\delta_\delta(t)\delta_\varepsilon(t)\delta_\rho(i)\\
&=&\sum_{\alpha\ldots\rho\in NC(n)}W_{nN}(\alpha,\beta)W_{nk}(\gamma,\delta)W_{nN}(\varepsilon,\rho)\sum_{ist}\delta_{\alpha\vee\rho}(i)\delta_{\beta\vee\gamma}(s)\delta_{\delta\vee\varepsilon}(t)\\
&=&\sum_{\alpha\ldots\rho\in NC(n)}W_{nN}(\alpha,\beta)W_{nk}(\gamma,\delta)W_{nN}(\varepsilon,\rho)\cdot l^{|\alpha\vee\rho|}k^{|\beta\vee\gamma|}k^{|\delta\vee\varepsilon|}
\end{eqnarray*}

Let us examine now the asymptotic regime $k=sN,l=tN,N\to\infty$ in the statement. We use here two standard facts from \cite{bc2}, namely the fact that in the $N\to\infty$ limit the Gram and Weingarten matrices are concentrated on the diagonal, and the fact that we have $|\pi\vee\sigma|\leq\frac{|\pi|+|\sigma|}{2}$, with equality when $\pi=\sigma$. We obtain, as in \cite{bc2}:
\begin{eqnarray*}
c_n
&\simeq&\sum_{\alpha,\gamma,\varepsilon\in NC(n)}N^{-|\alpha|}k^{-|\gamma|}N^{-|\varepsilon|}\cdot l^{|\alpha\vee\varepsilon|}k^{|\alpha\vee\gamma|}k^{|\gamma\vee\varepsilon|}\\
&\simeq&\sum_{\alpha,\gamma,\varepsilon\in NC(n)}N^{-|\alpha|-|\gamma|-|\varepsilon|+|\alpha\vee\varepsilon|+|\alpha\vee\gamma|+|\gamma\vee\varepsilon|}\cdot s^{-|\gamma|+|\alpha\vee\gamma|+|\gamma\vee\varepsilon|}\cdot t^{|\alpha\vee\varepsilon|}\\
&\simeq&\sum_{\alpha\in NC(n)}(st)^{|\alpha|}
\end{eqnarray*}

We recognize at right the well-known formula for the moments of the free Poisson law of parameter $st$, cf. \cite{nsp}, and this finishes the proof.
\end{proof}

As a conclusion, with Theorem 1.6 and Theorem 2.9 in hand, and by using the well-known fact that Poisson ($st$) $\to$ free Poisson ($st$) is indeed a liberation, in the sense of probability theory \cite{vdn}, we can now state that $\widetilde{S}_N\to\widetilde{S}_N^+$ is a ``correct'' liberation.

We should perhaps comment a bit more on all this. The early attempts of liberation go back to Brown \cite{bro} and McClanahan \cite{mcc}, who studied a certain compact semigroup $U_N^{nc}$. Later on, Wang found in \cite{wa1}, \cite{wa2} the correct liberations of $O_N,U_N,S_N$. When looking at more complicated groups, the situation can be quite unobvious, and a ``liberation criterion'' is needed. And the answer here is provided by the Bercovici-Pata bijection \cite{bpa}, which establishes a correspondence between ``classical'' and ``free'' measures.

Technically speaking, the bijection is between infinitely divisible measures $\mu$, and freely infinite divisible measures $\eta$, and the construction of the correspondence is very simple, via the requirement ``the cumulants of $\mu$ are the free cumulants of $\eta$''. See \cite{bpa}, \cite{nsp}.

Now back to quantum groups, it was realized since \cite{bc1}, \cite{bc2} that the correct ``liberation criterion'' should be the Bercovici-Pata bijection for the laws of truncated characters, in the $N\to\infty$ limit. The paper \cite{bbc} successfully used this criterion for finding the correct liberation of the hyperoctahedral $H_N$: indeed, there are two natural candidates for such a liberation $H_N^+$, one satisfying the criterion, and the other one not. We refer to \cite{bbc} for the full story here, and to \cite{bsp}, \cite{fre}, \cite{rwe} for further developments of this idea.

\section{Discrete isometries}

In what follows we discuss various analogues, discrete and continuous, of the above results about $S_N$. We will focus on 5 ``fundamental'' groups, namely the orthogonal, unitary and bistochastic groups $O_N,U_N,B_N$, the hyperoctahedral group $H_N=\mathbb Z_2\wr S_N$, and the group $K_N=\mathbb T\wr S_N$ consisting of permutation matrices with entries in $\mathbb T$.

As a first remark, we have inclusions between these groups, as follows:

\begin{proposition} 
We have inclusions between the $6$ fundamental groups,
$$\begin{matrix}
B_N&\subset&O_N&\subset&U_N\\
\\
\cup&&\cup&&\cup\\
\\
S_N&\subset&H_N&\subset&K_N
\end{matrix}$$
and in addition we have $S_N=B_N\cap H_N$, and $H_N=O_N\cap K_N$.
\end{proposition}

\begin{proof}
All the assertions are clear from definitions. Let us mention, in addition, that we have as well $O_N=<B_N,H_N>$, and $U_N=<O_N,K_N>$. See \cite{bsp}.
\end{proof}

The choice of these particular groups comes from the ``easy quantum group'' philosophy from \cite{bsp}, further developed in \cite{fre}, \cite{rwe}. We are of course mostly interested in $O_N,U_N$, but inspiration from $S_N$, and especially from $H_N,K_N$, will be of great use.

In this section we discuss the discrete case, concerning $S_N,H_N,K_N$. We use a global approach, with a parameter $x=1,2,\infty$, coming from \cite{bb+}, \cite{bv1}. Let $\mathbb Z_x$ be the group of $x$-th roots of unity in the complex plane, with the convention $\mathbb Z_\infty=\mathbb T$. We have then:

\begin{proposition}
Given $x\in\{1,2,\ldots,\infty\}$, let $\widetilde{H}_N^x$ be the semigroup of partial permutations $\sigma\in\widetilde{S}_N$ ``signed'' by elements of $\mathbb Z_x$, that is, $\sigma^\varepsilon(i)=\varepsilon(i)\cdot\sigma(i)$, with $\varepsilon(i)\in\mathbb Z_x$.
\begin{enumerate}
\item $\widetilde{H}_N^x$ contains the group $H_N^x=\mathbb Z_x\wr S_N$.

\item At $x<\infty$ we have $|\widetilde{H}_N^x|=\sum_{k=0}^Nk!\binom{N}{k}^2x^k$.

\item We have a semigroup embedding $u:\widetilde{H}_N^x\subset M_N(\mathbb Z_x\cup\{0\})$.

\item We have a set-theoretical embedding $\widetilde{H}_N^x\subset H_{2N}^x$.
\end{enumerate}
\end{proposition}

\begin{proof}
Observe that at $x=1$ we obtain the semigroup $\widetilde{S}_N$, and the standard results about it, from section 1 above. In general, the proof is similar. First, the fact that $\widetilde{H}_N^x$ is indeed a semigroup, with multiplication obtained by composing the partial permutations, and then multiplying the ``signs'', inside the group $\mathbb Z_x$, follows from definitions.

(1) This is clear from the functoriality of $\wr$, because $\widetilde{H}_N^x=\mathbb Z_x\wr\widetilde{S}_N$. In the matrix model picture, $H_N^x\subset M_N(\mathbb Z_x\cup\{0\})$ is the group of matrices having exactly one nonzero entry on each row and each column, and the assertion follows as well from (3) below.

(2) This formula, generalizing the counting formula in Proposition 1.2, is clear. As a side remark here, in the $x\to\infty$ limit, the cardinality becomes concentrated on the group part, $H_N^x$. Also, we do not know what happens in the $N\to\infty$ limit.

(3) This  is clear too, by using the coordinate functions $u_{ij}:\widetilde{H}_N^x\to\mathbb Z_x\cup\{0\}$ given by $u_{ij}(\sigma^\varepsilon)=\varepsilon(j)\delta_{i\sigma(j)}$. Observe that the image of the embedding consists of the matrices in the target having at most one nonzero entry, on each row and each column.

(4) This statement, generalizing Proposition 1.4, is once again clear. We can indeed use the same formula as there, in order to construct the embedding.
\end{proof}

Observe that the linear map $T=u(\sigma)$ constructed in (3) above is given by $T(e_i)=\varepsilon(i)e_{\sigma(i)}$, where $\{e_1,\ldots,e_N\}$ is the standard basis of $\mathbb C^N$, with the convention $T(e_i)=0$ when $\sigma(i)$ is undefined. Thus, $\widetilde{H}_N^s$ is the semigroup of partial isometries of $\mathbb C^N$ which partially permute the coordinate axes, and rotate inside them by scalars in $\mathbb Z_x$. In the real cases, $x=1,2$, we can of course restrict these partial isometries to $\mathbb R^N$.

Now back to the fundamental groups in Proposition 3.1, observe that at  $x=1,2,\infty$ the group $H_N^x=\mathbb Z_x\wr S_N$ appearing in (1) above is given by $H_N^1=S_N$, $H_N^2=H_N$, $H_N^\infty=K_N$. Thus, we have now semigroup analoues of the groups $S_N,H_N,K_N$:

\begin{proposition}
The semigroup $\widetilde{H}_N^x$ specializes at $x=1,2,\infty$ as follows:
\begin{enumerate}
\item $\widetilde{S}_N=\widetilde{H}_N^1$ is the semigroup of partial permutations.

\item $\widetilde{H}_N=\widetilde{H}_N^2$ is the semigroup of signed partial permutations.

\item $\widetilde{K}_N=\widetilde{H}_N^\infty$ is the semigroup of partial permutations signed by elements of $\mathbb T$.
\end{enumerate}
\end{proposition}

\begin{proof}
The fact that at $x=1$ we obtain indeed $\widetilde{S}_N$ is clear. The assertions at $x=2,\infty$ are just particularizations of the main construction in Proposition 3.2.
\end{proof}

Summarizing, we have now semigroups $\widetilde{S}_N,\widetilde{H}_N,\widetilde{K}_N$, defined via a global approach. In what follows we will keep using the global approach, in terms of $x\in\{1,\ldots,\infty\}$, by keeping in mind that the $x=1,2,\infty$ cases are those which we are interested in. 

In order to liberate, recall that the quantum group $H_N^{x+}$ is defined by making its standard coordinates $u_{ij}$ subject to the following relations: $u=(u_{ij})$ and $u^t=(u_{ji})$ are unitaries, $u_{ij}u_{ij}^*=u_{ij}^*u_{ij}=p_{ij}$ (projections), and $u_{ij}^x=p_{ij}$. Here the last relation dissapears at $x=\infty$. We have $H_N^{1+}=S_N^+,H_N^{2+}=H_N^+,H_N^{\infty+}=K_N^+$. See \cite{bb+}, \cite{bv1}.

Regarding now semigroup case, we have here the following result:

\begin{proposition}
Let $C(\widetilde{H}_N^{x+})$ be the universal $C^*$-algebra generated by $N^2$ variables $u_{ij}$ such that $u_{ij}u_{ij}^*=u_{ij}^*u_{ij}=p_{ij}$ (projections), $p=(p_{ij})$ is submagic, and $u_{ij}^x=p_{ij}$.
\begin{enumerate}
\item $C(\widetilde{H}_N^{x+})$ is a bialgebra, with $\Delta(u_{ij})=\sum_ku_{ik}\otimes u_{kj}$ and $\varepsilon(u_{ij})=\delta_{ij}$.

\item We have an embedding $H_N^{x+}\subset\widetilde{H}_N^{x+}$, obtained by assuming that $p$ is magic.

\item We have an embedding $\widetilde{H}_N^x\subset\widetilde{H}_N^{x+}$, obtained assuming that the $u_{ij},u_{ij}^*$ commute.
\end{enumerate}
\end{proposition}

\begin{proof}
We use the fact that $aa^*=a^*a=p,bb^*=b^*b=q$, with $p,q$ orthogonal projections ($pq=qp=0$) implies $ab=ab^*=0$. This follows indeed from:
$$(ab)(ab)^*(ab)(ab)^*=(ab^*)(ab^*)^*(ab^*)(ab^*)^*=aqpqa^*=0$$

(1) The existence of $\varepsilon$, defined as in the statement, being clear, let us check the existence of $\Delta$. We set $U_{ij}=\sum_ku_{ik}\otimes u_{kj}$. By using the above observation, we get:
$$U_{ij}U_{ij}^*
=\sum_{kl}u_{ik}u_{il}^*\otimes u_{kj}u_{lj}^*
=\sum_ku_{ik}u_{ik}^*\otimes u_{kj}u_{kj}^*
=\sum_kp_{ik}\otimes p_{kj}$$

A similar computation shows that $U_{ij}^*U_{ij}$ is given by the same formula. Now since for $p=(p_{ij})$ magic, the matrix $P_{ij}=\sum_kp_{ik}\otimes p_{kj}$ is magic too, this finishes the verification of the first two conditions. The last condition, $U_{ij}^x=P_{ij}$, is routine to check as well:
\begin{eqnarray*}
U_{ij}^x
&=&\sum_{k_1\ldots k_x}u_{ik_1}\ldots u_{ik_x}\otimes u_{k_1j}\ldots u_{k_xj}\\
&=&\sum_ku_{ik}\ldots u_{ik}\otimes u_{kj}\ldots u_{kj}\\
&=&\sum_ku_{ik}^x\otimes u_{kj}^x=P_{ij}
\end{eqnarray*}

(2) This is clear from the definitions of $H_N^{x+}$ and $\widetilde{H}_N^{x+}$, because the unitarity conditions $uu^*=u^*u=1$ and $u^t\bar{u}=\bar{u}u^t=1$ tell us precisely that $p$ must be magic.

(3) This follows from Gelfand duality, by observing that $\widetilde{S}_N$ is part of the spectrum, and then by using $u_{ij}^x=$ projection, which in the commutative setting reads $Im(u_{ij})\subset\mathbb Z_x\cup\{0\}$, in order to reconstruct the whole semigroup $\widetilde{H}_N^x=\mathbb Z_x\wr\widetilde{S}_N$ inside the spectrum. For details we refer to \cite{bv1} for the group case, the proof in the semigroup case being similar.
\end{proof}

Regarding now the probabilistic aspects, we first need an extension of the ``$\sigma=\alpha\beta\gamma$'' trick. In order to formulate a global statement, pick an exponent $\circ\in\{\emptyset,+\}$, set $\kappa=\sum_{ij}u_{ij}u_{ij}^*$, and consider the algebra $C(\widetilde{H}_N^{x\circ(k)})=C(\widetilde{H}_N^{x\circ})/<\kappa=k>$. We have then:

\begin{proposition}
For any exponent $\circ\in\{\emptyset,+\}$ we have a representation
$$\pi_k:C(\widetilde{H}_N^{x\circ(k)})\to C(H_N^{x\circ}\times H_k^{x\circ}\times H_N^{x\circ})\quad:\quad\pi_k(u_{ij})=\sum_{s,t\leq k}p_{is}\otimes q_{st}\otimes r_{tj}$$
which commutes with the Haar functionals in the classical case, and at $k=N$.
\end{proposition}

\begin{proof}
In the free case this is a straightforward extension of Proposition 2.8 above. The only fact that needs to be checked is that the matrix on the right in the statement, say $U_{ij}$, produces indeed a representation of $C(\widetilde{H}_N^{x+(k)})$. For this purpose, we will use several times the formulae $ab=ab^*=0$, from the proof of Proposition 3.4 above. We have:
$$U_{ij}U_{ij}^*
=\sum_{s,t\leq k}\sum_{v,w\leq k}p_{is}p_{iv}^*\otimes q_{st}q_{vw}^*\otimes r_{tj}r_{wj}^*
=\sum_{s,t\leq k}p_{is}p_{is}^*\otimes q_{st}q_{st}^*\otimes r_{tj}r_{tj}^*$$

A similar computation shows that $U_{ij}^*U_{ij}$ is given by the same formula. Now since the matrices $(p_{is}p_{is}^*)$, $(q_{st}q_{st}^*)$, $(r_{tj}r_{tj}^*)$ are all magic, we deduce as in Proposition 2.8 (1) that the matrix on the right is submagic as well. Also, we have:
\begin{eqnarray*}
U_{ij}^x
&=&\sum_{s_1\ldots s_x\leq k}\sum_{t_1\ldots t_x\leq k}p_{is_1}\ldots p_{is_x}\otimes q_{s_1t_1}\ldots q_{s_xt_x}\otimes r_{s_1j}\ldots r_{s_xj}\\
&=&\sum_{s,t\leq k}p_{is}^x\otimes q_{st}^x\otimes r_{sj}^x=\sum_{s,t\leq k}p_{is}p_{is}^*\otimes q_{st}q_{st}^*\otimes r_{sj}r_{sj}^*
\end{eqnarray*}

We conclude that we have $U_{ij}^x=U_{ij}U_{ij}^*$. Let us verify now the fact that our representation vanishes indeed on the ideal $<\kappa=k>$. We have here:
\begin{eqnarray*}
\sum_{ij}U_{ij}U_{ij}^*
&=&\sum_{ij}\sum_{s,t\leq k}p_{is}p_{is}^*\otimes q_{st}q_{st}^*\otimes r_{tj}r_{tj}^*\\
&=&\sum_{s,t\leq k}1\otimes q_{st}q_{st}^*\otimes 1=k
\end{eqnarray*}

This finishes the construction of $\pi_k$, in the free case. In the classical case now, the existence of $\pi_k$ follows by restriction. Finally, the assertions regarding the Haar functional are all clear, the one regrarding the classical case this being obvious, and the one regarding the case $k=N$ being already known, from the proof of Proposition 2.8 (3).
\end{proof}

By using Proposition 3.5 we can now define complex probability measures $\mu_k^l=law(\chi_k^l)$, with $\chi_k^l=\pi_k(\chi_l)$, and we can formulate our probabilistic result, as follows:

\begin{theorem}
The operation $\widetilde{H}_N^x\to\widetilde{H}_N^{x+}$ is a liberation, in the sense that we have the Bercovici-Pata bijection for the measures $\mu_k^l$, in the $k=sN,l=tN,N\to\infty$ limit.
\end{theorem}

\begin{proof}
The proof here is similar to the proof of Theorem 1.6 and Theorem 2.9, by using this time more advanced integration technology, from \cite{bb+}, \cite{bbc}. 

Let us first discuss the case of $\widetilde{H}_N$. Here the truncated character $\chi_l$ is real, and we can write $\chi_l=\chi_l^+-\chi_l^-$, as in \cite{bbc}. For $p\in\mathbb N$, we have the following formula:
\begin{eqnarray*}
P_{\widetilde{H}_N^{(k)}}(\chi_l=p)
&=&\sum_{r\geq0}P_{\widetilde{H}_N^{(k)}}(\chi_l^+=p+r,\chi_l^-=r)\\
&=&\sum_{r\geq0}\left(\frac{1}{2}\right)^{p+2r}\binom{p+2r}{r}P_{\widetilde{H}_N^{(k)}}(\chi_l^++\chi_l^-=p+2r)\\
&=&\sum_{r\geq0}\left(\frac{1}{2}\right)^{p+2r}\binom{p+2r}{r}P_{\widetilde{S}_N^{(k)}}(\chi_l=p+2r)
\end{eqnarray*}

Here the binomial coefficient comes from selecting $r$ negative components among a total of $p+2r$ nonzero components, and the $1/2$ power comes from matching the signs of these $p+2r$ nonzero components. In the limit $k=sN,l=tN,N\to\infty$, we obtain:
$$P_{\widetilde{H}_N^{(k)}}(\chi_l=p)
\simeq\sum_{r\geq0}\left(\frac{1}{2}\right)^{p+2r}\binom{p+2r}{r}\frac{(ts)^{p+2r}}{e^{ts}(p+2r)!}
=e^{-t}\sum_{r\geq0}\frac{(st/2)^{p+2r}}{(p+r)!r!}$$

Here we have used the Poisson convergence result from Theorem 1.6 above. Now since the density is the same at $p$ and at $-p$, we conclude that we have:
$$\mu_k^l\simeq e^{-t}\sum_{p\in\mathbb Z}\sum_{r\geq 0}\frac{(st/2)^{|p|+2r}}{(|p|+r)!r!}\delta_p$$

More generally now, consider the semigroup $\widetilde{H}_N^x$, with $x\in\mathbb N\cup\{\infty\}$ arbitrary. As explained in \cite{bb+}, the same argument applies, and leads to the following general asymptotic formula, where $\rho_x$ denotes the uniform measure on the $x$-th roots of unity:
$$\mu_k^l\simeq e^{-st}\sum_{n\geq0}\frac{(st)^n}{n!}\rho_x^{*n}$$

Thus we are led to the Bessel laws from \cite{bb+}, the parameter here being $st$. As explained in \cite{bb+}, by the Poisson convergence theorem, these Bessel laws appear as compound Poisson laws, with respect to $\rho_x$, and the following formula for $*$-moments is available:
$$\mu_k^l(X^{e_1}\ldots X^{e_n})=\sum_{\alpha\in P_x(e_1,\ldots,e_n)}(st)^{|\alpha|}$$

Here $e_1,\ldots,e_n\in\{1,*\}$ are exponents, and $P_x$ is the category of partitions associated to the group $H_N^x$, with legs colored by $\{1,*\}$, having the property that in each block, the number of legs colored 1 equals the number of legs colored $*$, modulo $x$.

In the free case now, the Weingarten arguments from the proof of Theorem 2.9 above extend as well, by replacing $NC$ with the category $NC_x=P_x\cap NC$. We obtain the following formula, known already from \cite{bb+} to hold in the group case, $k=N$:
$$\int_{H_N^{x+}\times H_k^{x+}\times H_N^{x+}}(\chi_k^l)^{e_1}\ldots(\chi_k^l)^{e_n}=\sum_{\alpha\in NC_x(e_1,\ldots,e_n)}(st)^{|\alpha|}$$

See \cite{bbc} for details at $x=2$, and \cite{bb+} for details in general. Now by comparing with the classical formula, we conclude that we have the Bercovici-Pata bijection. See \cite{bb+}.
\end{proof}

\section{Continuous isometries}

In this section we discuss the ``continuous extension'' of the results in the previous sections, with the permutation-modelled groups $S_N,H_N,K_N$ replaced by the bistochastic, orthogonal and unitary groups $B_N,O_N,U_N$. We recall that $B_N$ is by definition the group of orthogonal matrices having sum 1 on each row and column. See \cite{bsp}.

Our starting point will be the following definition:

\begin{definition}
$\widetilde{O}_N,\widetilde{U}_N,\widetilde{B}_N$ are semigroups of partial linear isometries of $\mathbb R^N,\mathbb C^N$,
\begin{eqnarray*}
\widetilde{O}_N&=&\{T:A\to B\ {\rm isometry}|A,B\subset\mathbb R^N\}\\
\widetilde{U}_N&=&\{T:A\to B\ {\rm isometry}|A,B\subset\mathbb C^N\}\\
\widetilde{B}_N&=&\{T:A\to B\ {\rm stochastic\ isometry}|A,B\subset\mathbb R^N\}
\end{eqnarray*}
with the usual composition operation, $T'T:T^{-1}(A'\cap B)\to T'(A'\cap B)$.
\end{definition} 

Here we call a partial isometry $T\in\widetilde{O}_N$ ``stochastic'' if $T\xi_A=\xi_B$, where $\xi_A,\xi_B$ are the orthogonal projections on $A,B$ of the all-one vector $\xi=(1,\ldots,1)^t$.

As a first remark, $\widetilde{O}_N,\widetilde{U}_N$ are indeed semigroups, with respect to the operation in the statement. The same holds for $\widetilde{B}_N$, and this is best seen in the matrix model picture:

\begin{proposition}
We have an embedding $\widetilde{U}_N\subset M_N(\mathbb C)$, obtained by completing maps $T:A\to B$ into linear maps $U:\mathbb C^N\to\mathbb C^N$, by setting $U_{|A^\perp}=0$. Moreover:
\begin{enumerate}
\item This embedding makes $\widetilde{U}_N$ correspond to the set of matrix-theoretic partial isometries, i.e. to the matrices $U\in M_N(\mathbb C)$ satisfying $UU^*U=U$.

\item The semigroup operation on $\widetilde{U}_N$ corresponds in this way to the semigroup operation for matrix-theoretic partial isometries, $U\circ V=U(U^*U\wedge VV^*)V$.
\end{enumerate}
This embedding restricts to embeddings $\widetilde{O}_N,\widetilde{B}_N\subset M_N(\mathbb R)$, having as image the matrices $U\in M_N(\mathbb R)$ satisfying  $UU^tU=U$, respectively $UU^tU=U$ and $U\xi=UU^t\xi$.
\end{proposition}

\begin{proof}
All the assertions are elementary. For $C=A,B$ let $I_C:C\subset\mathbb R^N$ be the inclusion, and $P_C:\mathbb R^N\to C$ be the projection. The correspondence $T\leftrightarrow U$ is then given by:
$$\xymatrix@C=20mm{A\ar[r]^T&B\ar[d]^{I_B}\\\mathbb R^N\ar[u]^{P_A}\ar[r]_U&\mathbb R^N}\qquad\qquad
\xymatrix@C=20mm{A\ar[r]^T\ar[d]_{I_A}&B\\\mathbb R^N\ar[r]_U&\mathbb R^N\ar[u]_{P_B}}$$

The fact that the composition $U\circ V$ is indeed a partial isometry comes from the fact that the projections $U^*U$ and $VV^*$ are absorbed when performing the product:
$$U(U^*U\wedge VV^*)V\cdot V^*(U^*U\wedge VV^*)U^*\cdot U(U^*U\wedge VV^*)V=U(U^*U\wedge VV^*)V$$

Together with a few other standard facts, this finishes the proof for $\widetilde{U}_N$. The assertion about $\widetilde{O}_N$ follows from the one for $\widetilde{U}_N$. Finally, regarding $\widetilde{B}_N$, we have:
$$T\xi_A=\xi_B\iff TP_A\xi=P_B\xi\iff U\xi=UU^t\xi$$

We therefore have an embedding $\widetilde{B}_N\subset M_N(\mathbb R)$ as in the statement. In order to check now that $\widetilde{B}_N$ is indeed a semigroup, something that we have not done yet, observe first that $U\xi=UU^t\xi$ implies $U^tU\xi=U^tUU^t\xi=U^t\xi$. Thus, for $U,V\in\widetilde{B}_N$ we have:
\begin{eqnarray*}
(U\circ V)(U\circ V)^t\xi
&=&U(U^tU\wedge VV^t)V\cdot V^t(U^tU\wedge VV^t)U^t\xi\\
&=&U(U^tU\wedge VV^t)U^t\xi\\
&=&U(U^tU\wedge VV^t)U^tU\xi\\
&=&U(U^tU\wedge VV^t)\xi
\end{eqnarray*}

On the other hand, once again since projections are absorbed, we have as well:
$$(U\circ V)\xi=U(U^tU\wedge VV^t)V\xi=U(U^tU\wedge VV^t)VV^t\xi=U(U^tU\wedge VV^t)\xi$$

We conclude that $(U\circ V)\xi=(U\circ V)(U\circ V)^t\xi$, and so $\widetilde{B}_N$ is indeed a semigroup.
\end{proof}

Regarding now various functoriality issues, we first have:

\begin{proposition} 
We have group and semigroup inclusions as follows,
$$\begin{matrix}
B_N&\subset&O_N&\subset&U_N\\
\\
\cup&&\cup&&\cup\\
\\
S_N&\subset&H_N&\subset&K_N
\end{matrix}\qquad:\qquad
\begin{matrix}
\widetilde{B}_N&\subset&\widetilde{O}_N&\subset&\widetilde{U}_N\\
\\
\cup&&\cup&&\cup\\
\\
\widetilde{S}_N&\subset&\widetilde{H}_N&\subset&\widetilde{K}_N
\end{matrix}$$
with the groups on the left being embedded into the semigroups on the right.
\end{proposition}

\begin{proof}
The group inclusions are already known, from Proposition 3.1 above. Regarding now the semigroup inclusions, the lower row comes from the $x=1,2,\infty$ specialization at the end of section 3, and the upper row comes from Definition 4.1. Regarding now the vertical semigroup inclusions, these are once again clear from the definition of the various semigroups involved, because we can make act signed partial permutations on the corresponding linear spaces spanned by coordinates. Finally, the last assertion is already known for $S_N,H_N,K_N$, is clear from definitions for $O_N,U_N$, and comes from the fact that a matrix $U\in O_N$ has sum 1 on each row and column precisely when $U\xi=\xi$.
\end{proof}

In relation now to the formula in Proposition 4.2 (2), observe that all six compositions $B_N,\widetilde{S}_N\subset\widetilde{B}_N\subset M_N(\mathbb R)$ and $O_N,\widetilde{H}_N\subset\widetilde{O}_N\subset M_N(\mathbb R)$ and $U_N,\widetilde{K}_N\subset\widetilde{U}_N\subset M_N(\mathbb C)$ are semigroup maps, with respect to the usual multiplication of the $N\times N$ matrices.

Observe also that we have set-theoretic embeddings $\widetilde{B}_N\subset B_{2N}$, $\widetilde{O}_N\subset O_{2N}$, $\widetilde{U}_N\subset U_{2N}$, that can be obtained by adapting the formula in Proposition 1.4 above.

In general, the multiplication formula $U\circ V=U(U^*U\wedge VV^*)V$ in Proposition 4.2 (2) is of course unavoidable. In view of some forthcoming liberation purposes, we would need a functional analytic interpretation of it. We have here the following result:

\begin{proposition}
$C(\widetilde{U}_N)$ is the universal commutative $C^*$-algebra generated by the entries of a $N\times N$ matrix $u=(u_{ij})$ satisfying $uu^*u=u$, with comultiplication
$$(id\otimes\Delta)u=u_{12}(p_{13}\wedge q_{12})u_{13}=\lim_{n\to\infty}\underbrace{UU^*\ldots U^*U}_{2n+1\ {\rm terms}}$$
where $p=uu^*,q=u^*u$ and $U_{ij}=\sum_ku_{ik}\otimes u_{kj}$. Dividing by the relations $u_{ij}=u_{ij}^*$ produces $C(\widetilde{O}_N)$, and further dividing by the relations $u\xi=uu^t\xi$ produces $C(\widetilde{B}_N)$.
\end{proposition}

\begin{proof}
The various presentation results follow from Proposition 4.2, by using the Gelfand and Stone-Weierstrass theorems. Let us find now the comultiplication of $C(\widetilde{U}_N)$. This is the map given by $\Delta(u_{ij})=\Phi^{-1}(L_{ij})$, where $\Phi:C(\widetilde{U}_N)\otimes C(\widetilde{U}_N)\to C(\widetilde{U}_N\times \widetilde{U}_N)$ is the canonical isomorphism, and where $L_{ij}(U,V)=(U\circ V)_{ij}$. In order to write now this map $L_{ij}$ in tensor product form, we can use the formula $P\wedge Q=\lim_{n\to\infty}(PQ)^n$. More precisely, with $P=VV^*$ and $Q=U^*U$, we obtain the following formula:
$$(U\circ V)_{ij}=\sum_{kl}U_{ik}(P\wedge Q)_{kl}V_{lj}=\lim_{n\to\infty}\sum_{kl}U_{kl}(PQ)^n_{kl}V_{lj}$$

With $a_0=k,a_{2n}=l$, and by expanding the product, we obtain:
\begin{eqnarray*}
(U\circ V)_{ij}
&=&\lim_{n\to\infty}\sum_{a_0\ldots a_{2n}}U_{ia_0}P_{a_0a_1}Q_{a_1a_2}\ldots P_{a_{2n-2}a_{2n-1}}Q_{a_{2n-1}a_{2n}}V_{a_{2n}j}\\
&=&\lim_{n\to\infty}\sum_{a_0\ldots a_{2n}}U_{ia_0}Q_{a_1a_2}\ldots Q_{a_{2n-1}a_{2n}}\cdot P_{a_0a_1}\ldots P_{a_{2n-2}a_{2n-1}}V_{a_{2n}j}
\end{eqnarray*}

Now by getting back to $\Delta(u_{ij})=\Phi^{-1}(L_{ij})$, with $L_{ij}(U,V)=(U\circ V)_{ij}$, we conclude that we have the following formula, with $p=uu^*$ and $q=u^*u$:
$$\Delta(u_{ij})=\lim_{n\to\infty}\sum_{a_0\ldots a_{2n}}u_{ia_0}q_{a_1a_2}\ldots q_{a_{2n-1}a_{2n}}\otimes p_{a_0a_1}\ldots p_{a_{2n-2}a_{2n-1}}u_{a_{2n}j}$$

Let us expand now both matrix products $p=uu^*$ and $q=u^*u$. In terms of the element $U_{ij}=\sum_ku_{ik}\otimes u_{kj}$ in the statement, the sum on the right, say $S_{ij}^{(n)}$, becomes:
\begin{eqnarray*}
S_{ij}^{(n)}
&=&\sum_{a_s}u_{ia_0}(u^*u)_{a_1a_2}\ldots(u^*u)_{a_{2n-1}a_{2n}}\otimes(uu^*)_{a_0a_1}\ldots(uu^*)_{a_{2n-2}a_{2n-1}}u_{a_{2n}j}\\
&=&\sum_{a_sb_sc_s}u_{ia_0}u_{b_1a_1}^*u_{b_1a_2}\ldots u_{b_na_{2n-1}}^*u_{b_na_{2n}}\otimes u_{a_0c_1}u_{a_1c_1}^*\ldots u_{a_{2n-2}c_n}u_{a_{2n-1}c_n}^*u_{a_{2n}j}\\
&=&\sum_{b_sc_s}U_{ic_1}U_{b_1c_1}^*U_{b_1c_2}\ldots U_{b_nc_n}^*U_{b_nj}=(\underbrace{UU^*\ldots U^*U}_{2n+1\ {\rm terms}})_{ij}
\end{eqnarray*}

Thus we have obtained the second formula in the statement. Regarding now the first formula, observe that we have $U=u_{12}u_{13}$. This gives:
\begin{eqnarray*}
\underbrace{UU^*\ldots U^*U}_{2n+1\ {\rm terms}}
&=&(u_{12}u_{13})(u_{13}^*u_{12}^*)\ldots(u_{13}^*u_{12}^*)(u_{12}u_{13})\\
&=&u_{12}(u_{13}u_{13}^*)(u_{12}^*u_{12})\ldots(u_{13}u_{13}^*)(u_{12}^*u_{12})u_{13}\\
&=&u_{12}p_{13}q_{12}\ldots p_{13}q_{12}u_{13}
\end{eqnarray*}

Now since the product on the right converges in the $n\to\infty$ limit to $u_{12}(p_{13}\wedge q_{12})u_{13}$, this gives the first formula in the statement as well, and we are done.
\end{proof}

Observe that if we further assume that $u$ is unitary, or that its entries satisfy the condition $u_{ij}u_{ij}^*=p_{ij}$  (projection) with $p=(p_{ij})$ magic, then $UU^*U=U$, so the convergence in the formula of $\Delta$ is stationary, and we obtain $\Delta(u_{ij})=U_{ij}$. Thus, we can recover in this way the fact that both the inclusions $U_N,\widetilde{K}_N\subset\widetilde{U}_N\subset M_N(\mathbb C)$ are semigroup maps, with respect to the usual multiplication of the $N\times N$ matrices. We will be back to this observation, with full details directly in the free case, in Proposition 4.7 below.

Let us construct now the liberations. We have here the following definition:

\begin{definition}
To any $N\in\mathbb N$ we associate the following algebras,
\begin{eqnarray*}
C(\widetilde{U}_N^+)&=&C^*((u_{ij})_{i,j=1,\ldots,N}|uu^*=p,\bar{u}u^t=p',{\rm projections})\\
C(\widetilde{O}_N^+)&=&C^*((u_{ij})_{i,j=1,\ldots,N}|u_{ij}=u_{ij}^*,uu^t=p, {\rm projection})\\
C(\widetilde{B}_N^+)&=&C^*((u_{ij})_{i,j=1,\ldots,N}|u_{ij}=u_{ij}^*,uu^t=p,{\rm projection},u\xi=uu^t\xi)
\end{eqnarray*}
and we call the underlying objects $\widetilde{U}_N^+,\widetilde{O}_N^+,\widetilde{B}_N^+$ spaces of quantum partial isometries.
\end{definition}

As a first observation, due to the presentation results in Poroposition 4.4, we have inclusions $\widetilde{U}_N\subset\widetilde{U}_N^+$, $\widetilde{O}_N\subset\widetilde{O}_N^+$, $\widetilde{B}_N\subset\widetilde{B}_N^+$. We have as well a liberated version of Proposition 4.3, or rather of the last assertion there, the rest being already known.

These functoriality statements are best summarized as follows:

\begin{proposition}
By assuming that ``all projections'' in Definition 4.5 are $1$,
$$\begin{matrix}
\widetilde{B}_N^+&\subset&\widetilde{O}_N^+&\subset&\widetilde{U}_N^+\\
\\
\cup&&\cup&&\cup\\
\\
\widetilde{B}_N&\subset&\widetilde{O}_N&\subset&\widetilde{U}_N
\end{matrix}
\qquad\mathop{\longrightarrow}^{p=p'=q=q'=1}\qquad
\begin{matrix}
B_N^+&\subset&O_N^+&\subset&U_N^+\\
\\
\cup&&\cup&&\cup\\
\\
B_N&\subset&O_N&\subset&U_N
\end{matrix}$$
where $q=u^*u$ and $q'=u^t\bar{u}$, we obtain the compact quantum groups on the right.
\end{proposition}

\begin{proof}
We recall that $O_N^+,U_N^+$ are the quantum groups constructed by Wang in \cite{wa1}, \cite{wa2}, and $B_N,B_N^+$ are the bistochastic group and its liberation, introduced in \cite{bsp}.

Let us check now that the algebras in Definition 4.5, when divided by the relations $p=p'=q=q'=1$, produce the algebras of continuous functions on these 6 quantum groups. Here we set, by definition, $p'=p$ and $q'=q$ in the 4 real cases.

In the case of $\widetilde{U}_N^+$, the relations $p=p'=q=q'=1$ read $uu^*=\bar{u}u^t=u^*u=u^t\bar{u}=1$, so we deduce that both $u,u^t$ are unitaries, and we are done. In the $\widetilde{O}_N^+$ case the extra relations are $u_{ij}=u_{ij}^*$, so we obtain the subgroup $O_N^+\subset U_N^+$. Also, in the $\widetilde{B}_N^+$ case the further extra relations are $u\xi=\xi$, so we obtain the subgroup $B_N^+\subset O_N^+$. 

The remaining assertions are now clear, because the subgroups of $U_N^+,O_N^+,B_N^+$ obtained by using the commutation relations $ab=ba$ are $U_N,O_N,B_N$. See \cite{bsp}.
\end{proof}

As a last functorial observation, recall from \cite{bsp} that we have isomorphisms $B_N\simeq O_{N-1}$ and $B_N^+\simeq O_{N-1}^+$, obtained by decomposing the fundamental representation $u=v+1$, by using the fixed vector $\xi$. In the semigroup case we don't have analogues of such results, as one can see by carefully examining the inclusions $\widetilde{B}_2\subset\widetilde{O}_2$ and $\widetilde{B}_2^+\subset\widetilde{O}_2^+$.

Let us discuss now the multiplicative structure. We have here:

\begin{proposition}
$\widetilde{U}_N^+,\widetilde{O}_N^+,\widetilde{B}_N^+$ have (non-associative) multiplications given by
$$(id\otimes\Delta)u=u_{12}(p_{13}\wedge q_{12})u_{13}=\lim_{n\to\infty}\underbrace{UU^*\ldots U^*U}_{2n+1\ {\rm terms}}$$
where $p=uu^*,q=u^*u$ and $U_{ij}=\sum_ku_{ik}\otimes u_{kj}$. The embeddings $\widetilde{B}_N,B_N^+,\widetilde{S}_N^+\subset\widetilde{B}_N^+$ and $\widetilde{O}_N,O_N^+,\widetilde{H}_N^+\subset\widetilde{O}_N^+$ and $\widetilde{U}_N,U_N^+,\widetilde{K}_N^+\subset\widetilde{U}_N^+$ commute with the multiplications.
\end{proposition}

\begin{proof}
First of all, the equality between the two matrices on the right in the statement follows as in the proof of Proposition 4.4. Let us call $W=(W_{ij})$ this matrix.

In order to check that $\Delta(u_{ij})=W_{ij}$ defines indeed a morphism, we must verify that $W=(W_{ij})$ satisfies the conditions in Definition 4.5. In the unitary case, we have:
\begin{eqnarray*}
WW^*W
&=&u_{12}(p_{13}\wedge q_{12})u_{13}\cdot u_{13}^*(p_{13}\wedge q_{12})u_{12}^*\cdot u_{12}(p_{13}\wedge q_{12})u_{13}\\
&=&u_{12}(p_{13}\wedge q_{12})p_{13}(p_{13}\wedge q_{12})p_{12}(p_{13}\wedge q_{12})u_{13}\\
&=&u_{12}(p_{13}\wedge q_{12})u_{13}=W
\end{eqnarray*}

The verification of the remaining condition $\bar{W}W^t\bar{W}=\bar{W}$ is similar, because $u\to\bar{u}$ transforms $U\to\bar{U}$, hence $W\to\bar{W}$. Now since when adding the relations $u_{ij}=u_{ij}^*$ we have $U=\bar{U}$, and so $W=\bar{W}$, we are done as well with the orthogonal case.

In the bistochastic case now, we use the following formula:
$$WW^t=u_{12}(p_{13}\wedge q_{12})u_{13}\cdot u_{13}^t(p_{13}\wedge q_{12})u_{12}^t=u_{12}(p_{13}\wedge q_{12})u_{12}^t$$

Observe that $u\xi=uu^t\xi$ implies $u^t\xi=u^tu\xi$ as well. We therefore obtain:
\begin{eqnarray*}
WW^t\xi
&=&u_{12}(p_{13}\wedge q_{12})u_{12}^t\xi\\
&=&u_{12}(p_{13}\wedge q_{12})u_{12}^tu_{12}\xi\\
&=&u_{12}(p_{13}\wedge q_{12})q_{12}\xi\\
&=&u_{12}(p_{13}\wedge q_{12})\xi
\end{eqnarray*}

On the other hand, we have as well the following computation:
\begin{eqnarray*}
W\xi
&=&u_{12}(p_{13}\wedge q_{12})u_{13}\xi\\
&=&u_{12}(p_{13}\wedge q_{12})u_{13}u_{13}^t\xi\\
&=&u_{12}(p_{13}\wedge q_{12})p_{13}\xi\\
&=&u_{12}(p_{13}\wedge q_{12})\xi
\end{eqnarray*}

Thus we have $WW^t\xi=W\xi$, and this finishes the proof in the bistochastic case.

Regarding now the last assertion, by functoriality it is enough to verify it in the unitary case. For $\widetilde{U}_N\subset\widetilde{U}_N^+$ this is clear. For $U_N^+\subset\widetilde{U}_N^+$ this is clear too, because with $p=q=1$ we obtain $(id\otimes\Delta)u=U$, which is the usual comultiplication formula for $C(U_N^+)$.

Thus, it remains to check that $\widetilde{K}_N^+\subset\widetilde{U}_N^+$ commutes with the multiplications. According to the definition of $\widetilde{K}_N^+$ from section 3 above, the relations which produce it, from the bigger space $\widetilde{U}_N^+$, are $u_{ij}u_{ij}^*=r_{ij}$ (projections), and $r=(r_{ij})$ submagic. We have:
\begin{eqnarray*}
(UU^*U)_{ij}
&=&\sum_{kl}U_{ik}U_{lk}^*U_{lj}\\
&=&\sum_{kl}\sum_{abc}u_{ia}u_{lb}^*u_{lc}\otimes u_{ak}u_{bk}^*u_{cj}\\
&=&\sum_{kl}\sum_au_{ia}r_{la}\otimes r_{ak}u_{aj}\\
&=&\sum_au_{ia}\otimes u_{aj}=U_{ij}
\end{eqnarray*}

Here we have used several times the observation, from the proof of Proposition 3.4 above, that $ab=ab^*=0$, for any $a,b$ distinct coordinates, on the same row or column.

We conclude from $UU^*U=U$ that when passing to $\widetilde{K}_N^+$ the comultiplication becomes $\Delta(u_{ij})=U_{ij}$. But this is the usual comultiplication of $\widetilde{K}_N^+$, and we are done.
\end{proof}

The fact that $\widetilde{U}_N^+$, while containing the semigroups $\widetilde{U}_N,U_N^+,\widetilde{K}_N^+$, it not itself a semigroup, might seem quite surprising. The first thought would be that $\widetilde{U}_N^+$, as defined above, is just too big. It seems impossible, however, to find a reasonable smaller semigroup, containing $\widetilde{U}_N,U_N^+,\widetilde{K}_N^+$. We will see in section 6 below that these issues dissapear when talking about half-liberations, with the construction of a semigroup $\widetilde{U}_N^\times$.

Let us discuss now probabilistic aspects. We will see that, while our spaces $\widetilde{U}_N^+,\widetilde{O}_N^+,\widetilde{B}_N^+$ are not semigroups, the Bercovici-Pata bijection criterion is satisfied for them. 

We use the method in section 3. We first need an extension of the ``$\sigma=\alpha\beta\gamma$'' trick. So, pick a group $G_N\in\{O_N,U_N,B_N\}$, pick as well an exponent $\circ\in\{\emptyset,+\}$, set $\kappa=\sum_{ij}u_{ij}u_{ij}^*$, and consider the algebra $C(\widetilde{G}_N^{\circ(k)})=C(\widetilde{G}_N^{\circ})/<\kappa=k>$. We have then:

\begin{proposition}
For $G_N\in\{O_N,U_N,B_N\}$ and $\circ\in\{\emptyset,+\}$ we have a representation
$$\pi_k:C(\widetilde{G}_N^{\circ(k)})\to C(G_N^\circ\times G_k^\circ\times G_N^\circ)\quad:\quad\pi_k(u_{ij})=\sum_{s,t\leq k}p_{is}\otimes q_{st}\otimes r_{tj}$$
which commutes with the Haar functionals at $k=N$.
\end{proposition}

\begin{proof}
In the classical case, denote by $\mathbb K=\mathbb R,\mathbb C$ the ground field. The first observation is that any partial isometry $T:A\to B$, with the spaces $A,B\subset\mathbb K^N$ having dimension $\dim(A)=\dim(B)=k$, decomposes as $T=UVW$, with $U,W\in G_N$ and $V\in G_k$:
$$\xymatrix@C=20mm{A\ar[r]^T\ar[d]_W&B\\\mathbb K^k\ar[r]_V&\mathbb K^k\ar[u]_U}$$

For $O_N,U_N$ this is indeed clear, and for $B_N$ this follows from the computations below.

We conclude that we have a surjection $\varphi:G_N\times G_k\times G_N\to\widetilde{G}_N^{(k)}$ mapping $(U,V,W)$ to the partial isometry $T:W^{-1}(\mathbb K^k)\to U(\mathbb K^k)$ given by $T(W^{-1}x)=U(Vx)$. By proceeding now as in the proof of Proposition 2.7 (2) above, we see that the transpose map $\pi=\varphi^*$ is the representation in the statement, and we are done with the classical case.

In the free case, this is a straightforward extension of Proposition 2.8 above. Let us first check that the matrix $U=(U_{ij})$ formed by the elements appearing on the right in the statement satisfies the partial isometry condition. We have:
\begin{eqnarray*}
(UU^*U)_{ij}
&=&\sum_{kl}U_{ik}U_{lk}^*U_{lj}\\
&=&\sum_{kl}\sum_{s,t\leq k}\sum_{v,w\leq k}\sum_{y,z\leq k}p_{is}p_{lv}^*p_{ly}\otimes q_{st}q_{vw}^*q_{yz}\otimes r_{tk}r_{wk}^*r_{zj}\\
&=&\sum_{s,t\leq k}\sum_{y,z\leq k}p_{is}\otimes q_{st}q_{yt}^*q_{yz}\otimes r_{zj}\\
&=&\sum_{s,z\leq k}p_{is}\otimes q_{sz}\otimes r_{zj}=U_{ij}
\end{eqnarray*}

Since $u_{ij}=u_{ij}^*$ implies $U_{ij}=U_{ij}^*$, this proves the partial isometry condition in the orthogonal case too. Regarding now the bistochastic condition, we first have:
\begin{eqnarray*}
(U\xi)_i
&=&\sum_jU_{ij}=\sum_j\sum_{s,t\leq k}p_{is}\otimes q_{st}\otimes r_{tj}\\
&=&\sum_{s,t\leq k}p_{is}\otimes q_{st}\otimes 1=\sum_{s\leq k}p_{is}\otimes 1\otimes 1
\end{eqnarray*}

We have as well the following computation:
\begin{eqnarray*}
(UU^t\xi)_i
&=&\sum_{jk}U_{jk}U_{ik}
=\sum_{jk}\sum_{s,t\leq k}\sum_{v,w\leq k}p_{js}p_{iv}\otimes q_{st}q_{vw}\otimes r_{tk}r_{wk}\\
&=&\sum_{s,t\leq k}\sum_{v\leq k}p_{iv}\otimes q_{st}q_{vt}\otimes 1
=\sum_{s\leq k}p_{is}\otimes 1\otimes 1
\end{eqnarray*}

Thus we have $U\xi=UU^t\xi$, as desired. Let us ckeck now that the representation that we have just constructed vanishes on the ideal $<\kappa=k>$. We have:
\begin{eqnarray*}
\sum_{ij}U_{ij}U_{ij}^*
&=&\sum_{ij}\sum_{s,t\leq k}\sum_{v,w\leq k}p_{is}p_{iv}^*\otimes q_{st}q_{vw}^*\otimes r_{tj}r_{wj}^*\\
&=&\sum_{s,t\leq k}1\otimes q_{st}q_{st}^*\otimes 1=k
\end{eqnarray*}

Thus we have a representation $\pi_k$ as in the statement. Finally, the last assertion is already known, from the proof of Proposition 2.8 (3).
\end{proof}

With the above result in hand, we can construct measures $\mu_k^l$ as in the discrete case, $\mu_k^l=law(\chi_k^l)$ with $\chi_k^l=\pi_k(\chi_l)$, and we have the following result:

\begin{theorem}
$\widetilde{G}_N\to\widetilde{G}_N^+$ with $G_N=O_N,U_N,B_N$ is a liberation, in the sense that we have the Bercovici-Pata bijection for $\mu_k^l$, in the $k=sN,l=tN,N\to\infty$ limit.
\end{theorem}

\begin{proof}
This follows by using standard integration technology, from \cite{bc1}, \cite{bsp}, \cite{csn}.

More precisely, the Weingarten computation in the proof of Theorem 2.9 above gives the following formula, in the $k=sN$, $l=tN$, $N\to\infty$ limit, where $D(n)\subset P(n)$ denotes the set of partitions associated to the quantum group $G_N^\circ$ under consideration:
$$\lim_{N\to\infty}\int_{G_N^\circ\times G_k^\circ\times G_N^\circ}(\chi_k^l)^n=\sum_{\alpha\in D(n)}(st)^{|\alpha|}$$

On the other hand, we know from \cite{bc1}, \cite{bsp}, \cite{csn} that the law of the truncated character $\chi_l$ is given by the following formula, in the $l=tN$, $N\to\infty$ limit:
$$\lim_{N\to\infty}\int_{G_N^\circ}(\chi_l)^n=\sum_{\alpha\in D(n)}t^{|\alpha|}$$

Since in the unitary case we have similar formulae for $*$-moments, we conclude that in the $k=sN$, $l=tN$, $N\to\infty$ limit, we have the following equality of distributions:
$$\lim_{N\to\infty}\mu_k^l=\lim_{N\to\infty}\mu_N^{sl}$$

With this observation in hand, the Bercovici-Pata bijection follows from the various results in \cite{bc1}, \cite{bsp}, \cite{csn}, and basically comes from the fact that, for the quantum groups under consideration, the corresponding category of partitions $D\subset P$ is stable by removing blocks. More precisely, we obtain in this way real and complex Gaussian variables, shifted real Gaussian variables, and their free analogues. See \cite{bc1}, \cite{bsp}, \cite{csn}.
\end{proof}

Summarizing, we have now liberation results for $\widetilde{S}_N,\widetilde{H}_N,\widetilde{K}_N,\widetilde{B}_N,\widetilde{O}_N,\widetilde{U}_N$. One interesting question is that of finding the exact unitary easy quantum groups, or perhaps even generalizations, for which such results hold. For some key ingredients in dealing with such questions, we refer to the recent work of Raum-Weber \cite{rwe} and Freslon \cite{fre}.

\section{Orthogonal half-liberation}

In the reminder of this paper we discuss the half-liberation question for the semigroups $\widetilde{S}_N,\widetilde{H}_N,\widetilde{K}_N,\widetilde{B}_N,\widetilde{O}_N,\widetilde{U}_N$. There are several questions to be solved here, first because in the unitary group case already the half-liberation operation is not unique, and second because in the semigroup case we have some specific semigroup issues as well.

In the orthogonal group case, the basic half-liberation $O_N\to O_N^*$ was introduced in our joint work with Speicher \cite{bsp}, and was systematically studied in our paper with Vergnioux \cite{bv2}. Later on, Bichon and Dubois-Violette found in \cite{bdu} an axiomatic approach to the half-liberation operation $G\to G^*$ in the general orthogonal case, $G\subset O_N$. 

In the unitary group case the situation is more complicated, because the half-liberation is not unique. For the unitary group itself, a first key proposal, $U_N\to U_N^*$, was made by Bhowmick, D'Andrea and Dabrowski \cite{bdd}. A second proposal, $U_N\to U_N^{**}$, was made by Bichon and Dubois-Violette in \cite{bdu}. As observed in \cite{bdd}, \cite{bdu}, some other definitions for a half-liberation of $U_N$ are possible. We will discuss these issues in section 6 below.

In this section we discuss the orthogonal case, concerning $\widetilde{S}_N,\widetilde{H}_N,\widetilde{B}_N,\widetilde{O}_N$. We first construct some ``pre-half-liberations'' of these semigroups, as follows:

\begin{proposition}
Consider the subspace $\widetilde{O}_N^*\subset\widetilde{O}_N^+$ obtained by making the standard coordinates $u_{ij}$ half-commute, $abc=cba$. We have then embeddings as follows:
$$\begin{matrix}
\widetilde{O}_N&\subset&\widetilde{O}_N^*&\subset&\widetilde{O}_N^+\\
\\
\cup&&\cup&&\cup\\
\\
O_N&\subset&O_N^*&\subset&O_N^+\
\end{matrix}$$
We have $\widetilde{S}_N=\widetilde{S}_N^+\cap\widetilde{O}_N^*$, and similar diagrams for $\widetilde{H}_N^*=\widetilde{H}_N^+\cap\widetilde{O}_N^*$ and $\widetilde{B}_N^*=\widetilde{B}_N\cap\widetilde{O}_N^*$.
\end{proposition}

\begin{proof}
All the functoriality assertions are clear. Regarding the equality $\widetilde{S}_N=\widetilde{S}_N^+\cap\widetilde{O}_N^*$, this comes from the fact the relations $abc=cba$ give $ab^2=b^2a$, so for idempotents ($p=p^2$) these relations collapse to the usual commutation relations $ab=ba$.
\end{proof}

As a first observation, the above statement contains a first subtlety appearing in the semigroup case, with the lack of an analogue of the result $B_N^*=B_N$. Recall indeed from \cite{bsp} that this latter equality holds, because from $abc=cba$ we obtain $ab=ba$, simply by summing over $c=u_{11},\ldots, u_{1N}$. In the semigroup case no such trick is available. This is to be related to the observations in section 4 above, regarding the lack of semigroup extensions of the well-known isomorphisms $B_N\simeq O_{N-1}$ and $B_N^+\simeq O_{N-1}^+$.

Let us discuss now the construction of the multiplication:

\begin{proposition}
$\widetilde{H}_N^*,\widetilde{B}_N^*,\widetilde{O}_N^*$ all have multiplications, as follows:
\begin{enumerate}
\item $\widetilde{H}_N^*$ is a semigroup, with $\Delta(u_{ij})=\sum_ku_{ik}\otimes u_{kj}$ and $\varepsilon(u_{ij})=\delta_{ij}$.

\item $\widetilde{B}_N^*,\widetilde{O}_N^*$ are stable under the multiplications of the spaces $\widetilde{B}_N^+,\widetilde{O}_N^+$.
\end{enumerate}
\end{proposition}

\begin{proof}
We use various formulae from sections 3 and 4 above:

(1) This is clear from the inclusion $\widetilde{H}_N^*\subset\widetilde{H}_N^+$, because $\widetilde{H}_N^+$ has semigroup structure given by the $\Delta,\varepsilon$ maps in the statement, and it is well-known, and easy to check, that such maps factorize when dividing by the half-commutation relations $abc=cba$.

(2) We first prove that the multiplication of $\widetilde{O}_N^+$ leaves invariant $\widetilde{O}_N^*$. So, recall from Proposition 4.7 above that with $U_{ij}=\sum_ku_{ik}\otimes u_{kj}$ we have $\Delta(u_{ij})=W_{ij}$, where:
$$W_{ij}=\lim_{n\to\infty}(\underbrace{UU^t\ldots U^tU}_{2n+1\ {\rm terms}})_{ij}$$

We must prove that if the standard coordinates $u_{ij}$ satisfy the half-commutation relations $abc=cba$, then these elements $W_{ij}$ satisfy these relations as well. In order to do so, we use the following alternative formula, which appeared in the proof of Proposition 4.4, and which can be deduced in the free case exactly as in the classical case:
$$W_{ij}=\lim_{n\to\infty}\sum_{a_sb_sc_s}u_{ia_0}u_{b_1a_1}u_{b_1a_2}\ldots u_{b_na_{2n-1}}u_{b_na_{2n}}\otimes u_{a_0c_1}u_{a_1c_1}\ldots u_{a_{2n-2}c_n}u_{a_{2n-1}c_n}u_{a_{2n}j}$$

Observe that on both sides of the tensor product sign we have an odd number of variables. Now let us consider a product of type $W_{ij}W_{kl}W_{st}$. This appears as a triple limit, $\lim_{n,m,p\to\infty}$, of a certain sum of products of basic tensors as the above ones. Our claim is that, by using $abc=cba$ for the coefficients $u_{ij}$, each basic summand in the formula of $W_{ij}W_{kl}W_{st}$ can be identified with a basic summand in the formula of $W_{st}W_{kl}W_{ij}$.

In order to prove this claim, we use the general theory in \cite{bv2}. As explained there, the various identities involving variables $u_{ij}$ that can be obtained by using the half-commutation relations $abc=cba$ are best understood in terms of diagrams, with the relations $abc=cba$ themselves corresponding to the diagram $\slash\hskip-2mm\backslash\hskip-1.7mm|\hskip1mm$. The point now is that, since the number of variables in each 1/3 of basic summand is of the form $x+x$, with $x$ odd, the diagram needed for performing the transformation $W_{ij}W_{kl}W_{st}\to W_{st}W_{kl}W_{ij}$ is indeed in the category generated by $\slash\hskip-2mm\backslash\hskip-1.7mm|\hskip1mm$, computed in \cite{bv2}, and we are done.

Finally, the result for $\widetilde{B}_N^*$ follows from the result for $\widetilde{O}_N^*$.
\end{proof}

Let us discuss now some functoriality issues:

\begin{proposition}
We have inclusions of spaces with multiplications, as follows:
$$\begin{matrix}
\widetilde{B}_N^*&\subset&\widetilde{O}_N^*\\
\\
\cup&&\cup\\
\\
\widetilde{S}_N&\subset&\widetilde{H}_N^*
\end{matrix}$$
The inclusions in Proposition 5.1 above commute as well with the multiplications.
\end{proposition}

\begin{proof}
All the assertions are clear, by using the fact that the semigroup structures of $\widetilde{S}_N^*=\widetilde{S}_N$ and of $\widetilde{H}_N^*$ come from semigroup embeddings $\widetilde{S}_N^*\subset\widetilde{S}_N^+$ and $\widetilde{H}_N^*\subset\widetilde{H}_N^+$, which can be composed with the multiplicative embeddings $\widetilde{S}_N^+\subset\widetilde{B}_N^+$ and $\widetilde{H}_N^+\subset\widetilde{O}_N^+$.
\end{proof}

In order to get now more insight into the structure of $\widetilde{H}_N^*,\widetilde{B}_N^*,\widetilde{O}_N^*$, and especially into the quite non-standard multiplication of $\widetilde{B}_N^*,\widetilde{O}_N^*$, we use the machinery developed by Bichon and Dubois-Violette in \cite{bdu}. As explained there, the half-liberation operation for arbitrary subgroups $H\subset O_N$ is quite difficult to axiomatize directly, and best here is to introduce an extra object, namely a closed subgroup $K\subset U_N$:
$$\xymatrix{&K\subset U_N\ar@{<->}[dl]\ar@{<->}[dr]&\\H\subset O_N\ar@{<->}[rr]&&H^*\subset O_N^*}$$

Here the upper left correspondence is of compexification/taking real part type, the upper right correspondence can be approached via several algebraic methods (crossed products, $2\times2$ matrix models, Tannaka duality), and the lower correspondence is the half-liberation/abelianization operation, the one that we are interested in.

Following \cite{bdu}, we call a compact semigroup $K\subset\widetilde{U}_N$ self-conjugate if $g=(g_{ij})\in K$ implies $\bar{g}=(\bar{g}_{ij})\in K$. In this case we have an automorphism $s:C(K)\to C(K)$ given by $s(u_{ij})=u_{ij}^*$, and hence we can form a crossed product $C(K)\rtimes\mathbb Z_2$. 

With these notations, we have the following result, adapted from \cite{bdu}:

\begin{proposition}
Assume that $K\subset\widetilde{U}_N$ is self-conjugate, and define $s:C(K)\to C(K)$ by $s(u_{ij})=u_{ij}^*$. Then the following two constructions produce the same algebra:
\begin{enumerate}
\item $A\subset C(K)\rtimes\mathbb Z_2$ is the subalgebra generated by the elements $v_{ij}=u_{ij}s$.

\item $A\subset M_2(C(K))$ is the subalgebra generated by the elements $v_{ij}=\begin{pmatrix}0&u_{ij}\\ \bar{u}_{ij}&0\end{pmatrix}$.
\end{enumerate}
In addition, we have a bialgebra quotient map $C(\widetilde{O}_N^*)\to A$. We write $A=C(H^\times)$, where the compact semigroup $H\subset\widetilde{O}_N$ is the spectrum of the bialgebra $A/<ab=ba>$. 
\end{proposition}

\begin{proof}
The first observation is that the matrices in (2) satisfy the half-commutation relations $abc=cba$, because they multiply according to the following rule:
$$\begin{pmatrix}0&x\\ \bar{x}&0\end{pmatrix}\begin{pmatrix}0&y\\ \bar{y}&0\end{pmatrix}\begin{pmatrix}0&z\\ \bar{z}&0\end{pmatrix}=\begin{pmatrix}x\bar{y}&0\\ 0&\bar{x}y\end{pmatrix}\begin{pmatrix}0&z\\ \bar{z}&0\end{pmatrix}=\begin{pmatrix}0&x\bar{y}z\\ \bar{x}y\bar{z}&0\end{pmatrix}$$

These matrices are as well self-adjoint, $v_{ij}=v_{ij}^*$. Let us verify now the fact that the global matrix $v=(v_{ij})$ satisfies the partial isometry condition. We have: 
\begin{eqnarray*}
(vv^tv)_{ij}
&=&\sum_{kl}v_{ik}v_{lk}v_{lj}=\sum_{kl}\begin{pmatrix}0&u_{ik}\\ \bar{u}_{ik}&0\end{pmatrix}\begin{pmatrix}0&u_{lk}\\ \bar{u}_{lk}&0\end{pmatrix}\begin{pmatrix}0&u_{lj}\\ \bar{u}_{lj}&0\end{pmatrix}\\
&=&\begin{pmatrix}0&\sum_{kl}u_{ik}\bar{u}_{lk}u_{lj}\\ \sum_{kl}\bar{u}_{ik}u_{lk}\bar{u}_{lj}&0\end{pmatrix}
=\begin{pmatrix}0&u_{ij}\\ \bar{u}_{ij}&0\end{pmatrix}=v_{ij}
\end{eqnarray*}

We conclude that, with $A$ being as in (2), we have a quotient map $C(\widetilde{O}_N^*)\to A$.

It remains to prove that, when $K\subset\widetilde{U}_N$ is a semigroup, assumed to be self-conjugate, this algebra is a bialgebra, and coincides with the algebra constructed in (1).

We recall that $C(K)\rtimes\mathbb Z_2$ is by definition the coalgebra $C(K)\otimes C(\mathbb Z_2)$, with product $(fs^i)(gs^j)=(fs^i(g))s^{i+j}$ and involution $(fs^i)^*=(s^i(f)^*)s^i$, where we use the notation $fs^i=f\otimes s^i$. Our claim is that we have an embedding, as follows:
$$C(K)\rtimes\mathbb Z_2\subset M_2(C(K))\qquad:\qquad u_{ij}s\to\begin{pmatrix}0&u_{ij}\\ \bar{u}_{ij}&0\end{pmatrix}$$

Indeed, such an embedding can be constructed as follows:
$$f\to\begin{pmatrix}f&0\\ 0&s(f)\end{pmatrix},\qquad
fs\to\begin{pmatrix}0&f\\ s(f)&0\end{pmatrix}$$

Now with the above claim in hand, we conclude that $A$ coincides with the algebra constructed in (1). As for the bialgebra assertion, this is clear from the picture coming from (1), because $A$ appears as subalgebra of the crossed product. See \cite{bdu}.
\end{proof}

We can use the above method for investigating $\widetilde{H}_N^*,\widetilde{O}_N^*$, as follows:

\begin{proposition}
We have maps as follows
$$\begin{matrix}
C(\widetilde{O}_N^*)&\to&C(\widetilde{U}_N)\rtimes\mathbb Z_2&\subset&M_2(C(\widetilde{U}_N))\\
\\
\cup&&\cup&&\cup\\
\\
C(\widetilde{H}_N^*)&\to&C(\widetilde{K}_N)\rtimes\mathbb Z_2&\subset&M_2(C(\widetilde{K}_N))
\end{matrix}$$
given by the formulae in Proposition 5.4 above.
\end{proposition}

\begin{proof}
Since we have $\overline{(UU^tU)}_{ij}=(\bar{U}U^t\bar{U})_{ij}$, the relation $UU^tU=U$ implies $\bar{U}U^*\bar{U}=\bar{U}$, and so $\widetilde{U}_N$ is self-conjugate. The fact that $\widetilde{K}_N$ is self-conjugate too is clear as well. We can therefore apply the constructions in Proposition 5.4 above, and we obtain maps as in the statement, with only the map at bottom left still needing to be checked.

Here we just need to check that the $2\times 2$ matrices in Proposition 5.4 (2) satisfy the extra relations $v_{ij}^2=p_{ij}$ (projection) and $p=(p_{ij})$ magic, defining $\widetilde{H}_N^*$. We have:
$$v_{ij}^2=\begin{pmatrix}0&u_{ij}\\\bar{u}_{ij}&0\end{pmatrix}^2=
\begin{pmatrix}|u_{ij}|^2&0\\ 0&|u_{ij}|^2\end{pmatrix}$$

Now by remembering that the standard coordinates $u_{ij}$ on the semigroup $\widetilde{K}_N$ satisfy $u_{ij}u_{ij}^*=p_{ij}$ (projection) and $p=(p_{ij})$ magic, this gives the result.
\end{proof}

As explained in \cite{bdu}, in the compact group case, the maps $C(O_N^*)\to C(U_N)\rtimes\mathbb Z_2$ and $C(H_N^*)\to C(K_N)\rtimes\mathbb Z_2$ are embeddings. This fact however is quite non-trivial, and its proof uses a number of ingredients, for instance several key facts regarding the projective version operation, $G\to PG$, which are not available in the semigroup case.

In order to overcome these issues, we can proceed as follows:

\begin{theorem}
Let $\widetilde{O}_N^\times\subset\widetilde{O}_N^*$ and $\widetilde{H}_N^\times\subset\widetilde{H}_N^*$ be the subspaces corresponding to the images of the maps $C(\widetilde{O}_N^*)\to C(\widetilde{U}_N)\rtimes\mathbb Z_2$ and $C(\widetilde{H}_N^*)\to C(\widetilde{K}_N)\rtimes\mathbb Z_2$ constructed above.
\begin{enumerate}
\item $\widetilde{O}_N^\times,\widetilde{H}_N^\times$ are semigroups, with multiplication coming from the crossed products.

\item $\widetilde{O}_N\subset\widetilde{O}_N^\times\subset\widetilde{O}_N^*$ and $\widetilde{H}_N\subset\widetilde{H}_N^\times\subset\widetilde{H}_N^*$ commute with the multiplications.
\end{enumerate}
\end{theorem}

\begin{proof}
We use the various formulae in Proposition 5.4 and Proposition 5.5 above.

(1) This assertion is clear from Proposition 5.4 above.

(2) We first discuss the orthogonal case. Observe that we have an inclusion $\widetilde{O}_N\subset\widetilde{O}_N^\times$ as in the statement. This inclusion can be constructed indeed as follows:
$$C(\widetilde{O}_N^\times)\subset C(\widetilde{U}_N)\rtimes\mathbb Z_2\to C(\widetilde{U}_N)\to C(\widetilde{O}_N)$$

Here the middle map is $id\otimes\varepsilon$, and at right we have the standard quotient map, coming from the embedding $\widetilde{O}_N\subset\widetilde{U}_N$. Indeed, the resulting composition is as follows:
$$v_{ij}\to u_{ij}s\to u_{ij}\to v_{ij}$$

In remains to prove that the inclusion $\widetilde{O}_N^\times\subset \widetilde{O}_N^*$ commute with the multiplications. We must check here the commutativity of the following diagram:
$$\xymatrix@C=20mm{C(\widetilde{O}_N^*)\ar[d]_\Delta\ar[r]^\pi&C(\widetilde{U}_N)\rtimes\mathbb Z_2\ar[d]^\Delta\\ C(\widetilde{O}_N^*)\otimes C(\widetilde{O}_N^*)\ar[r]_{\pi\otimes\pi}&(C(\widetilde{U}_N)\rtimes\mathbb Z_2)^{\otimes2}}$$

We verify this on the standard generators $v_{ij}$. First, with $U_{ij}=\sum_ku_{ik}\otimes u_{kj}$ and $V_{ij}=\sum_kv_{ik}\otimes v_{kj}$, we have the following computation:
$$(\pi\otimes\pi)V_{ij}=\sum_k\pi(v_{ik})\otimes\pi(v_{kj})=\sum_ku_{ik}s\otimes u_{kj}s=U_{ij}(s\otimes s)$$

Thus we have $(id\otimes\pi\otimes\pi)V=U(s\otimes s)$, and we deduce that we have:
\begin{eqnarray*}
(\pi\otimes\pi)\Delta(v_{ij})
&=&(\pi\otimes\pi)\lim_{n\to\infty}(\underbrace{VV^*\ldots V^*V}_{2n+1\ {\rm terms}})_{ij}\\
&=&\lim_{n\to\infty}(\underbrace{UU^*\ldots U^*U}_{2n+1\ {\rm terms}})_{ij}(s\otimes s)\\
&=&\Delta(u_{ij})\Delta(s)=\Delta(u_{ij}s)=\Delta(\pi(v_{ij}))
\end{eqnarray*}

Thus the above diagram commutes, and we are done with the orthogonal case. The assertion in the hyperoctahedral case follows by restriction.
\end{proof}

There are many questions in connection with the above results. Perhaps the most important one concerns the semigroup analogue of the key isomorphisms $PO_N^*=PU_N$ and $PH_N^*=PK_N$ from \cite{bv2}. It is quite unclear whether something can be said here.

\section{Unitary half-liberation}

In this section we discuss the half-liberation problem for the remaining semigroups, $\widetilde{K}_N,\widetilde{U}_N$. As already mentioned in section 5, in the unitary case the half-liberation is not unique. One proposal, involving the relations $ab^*c=cb^*a$, was made by  Bhowmick, D'Andrea and Dabrowski in \cite{bdd}. Another proposal, involving the relations $abc=cba$, was made by Bichon and Dubois-Violette in \cite{bdu}. More precisely, we have:

\begin{definition}
The compact quantum groups $U_N^{**}\subset U_N^*$, which appear as intermediate subgroups $U_N\subset U_N^{**}\subset U_N^*\subset U_N^+$, are constructed as follows:
\begin{enumerate}
\item $U_N^*$, via the relations $ab^*c=cb^*a$, for any $a,b,c\in\{u_{ij}\}$.

\item $U_N^*$, via the relations $abc=cba$, for any $a,b,c\in\{u_{ij},u_{ij}^*\}$.
\end{enumerate}
\end{definition}

We refer to \cite{bdd} for the general theory and for the noncommutative geometric meaning of the quantum group $U_N^*$, and to \cite{bdu} for the general theory of its subgroup $U_N^{**}$.

As a first observation, we can now half-liberate any group $G\subset U_N^+$ already possessing a liberation $G^+\subset U_N^+$, simply by setting $G^*=G^+\cap U_N^*,G^{**}=G^+\cap U_N^{**}$.

In order to discuss now semigroup extensions, consider first the group $H_N^x=\mathbb Z_s\wr S_N$ from section 3 above, with $x\in\{1,\ldots,\infty\}$, consider its liberation $H_N^{x+}$ constructed in \cite{bb+}, \cite{bv1}, and explained in section 3, and construct as above $H_N^{x*},H_N^{x**}$. We have then:

\begin{proposition}
Let $\widetilde{H}_N^{x**}\subset\widetilde{H}_N^{x*}\subset\widetilde{H}_N^{x+}$ by obtained by using the relations $ab^*c=cb^*a$ with $a,b,c\in\{u_{ij}\}$, and $abc=cba$ with $a,b,c\in\{u_{ij},u_{ij}^*\}$.
\begin{enumerate}
\item $\widetilde{H}_N^{x**},\widetilde{H}_N^{x*}$ are semigroups, with $\Delta(u_{ij})=\sum_ku_{ik}\otimes u_{kj}$ and $\varepsilon(u_{ij})=\delta_{ij}$.

\item We have an embedding $H_N^{x*}\subset\widetilde{H}_N^{x**}$, obtained by assuming that $p$ is magic.

\item At $x=1,2$ we obtain respectively $\widetilde{S}_N$ twice, and $\widetilde{H}_N^*$ twice.
\end{enumerate}
\end{proposition}

\begin{proof}
These are elementary statements, whose proof goes as follows:

(1) The existence of $\varepsilon,\Delta$ are both clear, because if the standard coordinates $u_{ij}$ satisfy $abc=cba$, then the elements $U_{ij}=\sum_ku_{ik}\otimes u_{kj}$ satisfy these relations too.

(2) Once again this is an elementary fact, which follows by using the same arguments as in the free case, from the proof of Proposition 3.4 (2) above.

(3) This is clear from definitions, because at $x=1,2$ the coordinates are real, and so the two constructions in the statement produce the usual (real) half-liberation.
\end{proof}

In the unitary case now, we have the following analogue of Definition 6.1:

\begin{proposition}
Let $\widetilde{U}_N^{**}\subset\widetilde{U}_N^*\subset\widetilde{U}_N^+$ by obtained by using the relations $ab^*c=cb^*a$ with $a,b,c\in\{u_{ij}\}$, and $abc=cba$ with $a,b,c\in\{u_{ij},u_{ij}^*\}$.
\begin{enumerate}
\item $\widetilde{U}_N^{**},\widetilde{U}_N^*$ are stable under the multiplication of $\widetilde{U}_N^+$.

\item We have an embedding $U_N\subset\widetilde{U}_N^{**}$, obtained by assuming that $u$ is unitary.

\item We have embeddings $\widetilde{K}_N^*\subset\widetilde{U}_N^*$ and $\widetilde{O}_N^*,\widetilde{K}_N^{**}\subset\widetilde{U}_N^{**}$. 
\end{enumerate}
\end{proposition}

\begin{proof}
We use various formulae and methods from sections 4 and 5 above:

(1) The proof here is similar to the one of Proposition 5.2 (2).

(2) This is clear from the discussion in the proof of Proposition 4.6 above.

(3) Once again, this is clear from the definitions of the various objets involved.
\end{proof}

We can apply now the methods in \cite{bdu}, in order to construct smaller half-liberations $K_N^\times\subset K_N^{**}$ and $U_N^\times\subset U_N^{**}$, which are semigroups. We use the following semigroup analogues of the key isomorphisms $U_N\simeq O_{2,N}$ and $K_N\simeq H_{2,N}$ from \cite{bdu}:

\begin{proposition}
The following hold:
\begin{enumerate}
\item For $A,B\in M_N(\mathbb R)$ we have $A+iB\in\widetilde{U}_N\iff\begin{pmatrix}A&B\\-B&A\end{pmatrix}\in\widetilde{O}_{2N}$.

\item Both $\widetilde{U}_{2,N}=\{U\in\widetilde{U}_{2N}|U=(^A_{-B}{\ }^B_A)\}$ and $\widetilde{O}_{2,N}=\widetilde{U}_{2,N}\cap\widetilde{O}_{2N}$ are semigroups. 

\item Also, both $\widetilde{K}_{2,N}=\widetilde{U}_{2,N}\cap\widetilde{K}_{2N}$ and $\widetilde{H}_{2,N}=\widetilde{U}_{2,N}\cap\widetilde{H}_{2N}$ are semigroups. 

\item We have isomorphisms $\widetilde{U}_N\simeq\widetilde{O}_{2,N}$ and $\widetilde{K}_N\simeq\widetilde{H}_{2,N}$, coming from (1).
\end{enumerate}
\end{proposition}

\begin{proof}
These assertions are all elementary, the proof being as follows:

(1) For $U=A+iB$ with $A,B$ real we have indeed the following formula, whose proof is routine, and which gives the result:
$$\begin{pmatrix}A&B\\-B&A\end{pmatrix}
\begin{pmatrix}A^t&-B^t\\B^t&A^t\end{pmatrix}
\begin{pmatrix}A&B\\-B&A\end{pmatrix}
=\begin{pmatrix}Re(UU^*U)&Im(UU^*U)\\-Im(UU^*U)&Re(UU^*U)\end{pmatrix}$$

(2) The fact that $\widetilde{U}_{2,N}$ is a semigroup follows from the fact that if we assume that both $U,V$ have the pattern $(^A_{-B}{\ }^B_A)$, then in the multiplication formula $U\circ V=U(U^*U\wedge VV^*)V$, all the matrices which appear have this pattern too. The $\widetilde{O}_{2,N}$ assertion is clear.

(3) This is clear, because we are intersecting here certain semigroups.

(4) The isomorphism $\widetilde{U}_N\simeq\widetilde{O}_{2,N}$ comes from (1). Regarding now the isomorphism $\widetilde{K}_N\simeq\widetilde{H}_{2,N}$, this follows by restriction, because both $\widetilde{K}_N\subset\widetilde{U}_N$ and $\widetilde{H}_{2,N}\subset\widetilde{O}_{2,N}$ appear as subsets of matrices having at most one nonzero entry, on each row and column.
\end{proof}

We have the following analogue of Proposition 5.5 above:

\begin{proposition}
We have maps as follows
$$\begin{matrix}
C(\widetilde{U}_N^{**})&\to&C(\widetilde{U}_{2,N})\rtimes\mathbb Z_2&\subset&M_2(C(\widetilde{U}_{2,N}))\\
\\
\cup&&\cup&&\cup\\
\\
C(\widetilde{K}_N^{**})&\to&C(\widetilde{K}_{2,N})\rtimes\mathbb Z_2&\subset&M_2(C(\widetilde{K}_{2,N}))
\end{matrix}$$
given by the formulae in Proposition 5.4 above.
\end{proposition}

\begin{proof}
Since both semigroups $\widetilde{U}_{2,N},\widetilde{K}_{2,N}$ are self-conjugate, we can apply the constructions in Proposition 5.4 above. We obtain in this way maps as follows:
$$\begin{matrix}
C(\widetilde{O}_{2N}^*)&\to&C(\widetilde{U}_{2,N})\rtimes\mathbb Z_2&\subset&M_2(C(\widetilde{U}_{2,N}))\\
\\
\cup&&\cup&&\cup\\
\\
C(\widetilde{H}_{2N}^*)&\to&C(\widetilde{K}_{2,N})\rtimes\mathbb Z_2&\subset&M_2(C(\widetilde{K}_{2,N}))
\end{matrix}$$

Here the fact that at bottom left we can replace the algebra $C(\widetilde{O}_{2N}^*)$ by the smaller algebra $C(\widetilde{H}_{2N}^*)$ comes from the diagram in Proposition 5.5 above.

We must prove that $C(\widetilde{O}_{2N}^*)\to C(\widetilde{U}_{2,N})\rtimes\mathbb Z_2$ can be factorized through $C(\widetilde{U}_N^{**})$. In order to do so, the first observation is that we have a surjective morphism of algebras $C(\widetilde{O}_{2N}^+)\to C(\widetilde{U}_N^+)$, as follows, where $Re(u)=\frac{1}{2}(u+\bar{u})$ and $Im(u)=\frac{1}{2i}(u-\bar{u})$:
$$v\to\begin{pmatrix}Re(u)&Im(u)\\-Im(u)&Re(u)\end{pmatrix}$$

Indeed, the verification of the partial isometry condition follows from the same formula as the one in the proof of Proposition 6.4 (1) above, with the transpose sign replaced by the involution, because the proof of that formula doesn't use commutativity.

Thus we have an embedding $\widetilde{U}_N^+\subset\widetilde{O}_{2N}^+$, extending the emdedding $\widetilde{U}_N\simeq\widetilde{O}_{2,N}\subset\widetilde{O}_{2N}$ from Proposition 6.4 above. Our claim now is that we have a diagram as follows:
$$\begin{matrix}
\widetilde{O}_{2N}&\subset&\widetilde{O}_{2N}^*&\subset&\widetilde{O}_{2N}^+\\
\\
\cup&&\cup&&\cup\\
\\
\widetilde{U}_N&\subset&\widetilde{U}_N^{**}&\subset&\widetilde{U}_N^+
\end{matrix}$$

Indeed, we have to check here that the subgroup $\widetilde{U}_N^{**}\subset\widetilde{U}_N^+\subset\widetilde{O}_{2N}^+$ is contained into the subgroup $\widetilde{O}_{2N}^*\subset\widetilde{O}_{2N}^+$. But this follows from the fact that the $abc=cba$ relations applied to the entries of $Re(u),Im(u)$ imply the $abc=cba$ relations for all elements $u_{ij},u_{ij}^*$.

Our claim now is that we have a diagram as follows, where $C(\widetilde{O}_{2,N}^\times)$ is by definition the image of the map $C(\widetilde{O}_{2N}^*)\to C(\widetilde{U}_{2,N})\rtimes\mathbb Z_2$ constructed above:
$$\begin{matrix}
\widetilde{O}_{2,N}^\times
&\subset&\widetilde{U}_N^{**}
&\subset&\widetilde{O}_{2N}^*\\
\\
\cup&&\cup&&\cup\\
\\
\widetilde{O}_{2,N}&\simeq&\widetilde{U}_N&\subset&\widetilde{O}_{2N}
\end{matrix}$$

Indeed, the rectangle on the right is the rectangle on the left of the previous diagram. Regarding now the rectangle on the left, the bottom isomorphism is the one in Proposition 6.4 (4), the inclusion on the left comes by using the counit, as in the proof of Theorem 5.6 (2) above, and the inclusion on top follows by checking the extra relations by using the $2\times 2$ matrix model, as in the proof of Proposition 5.5 (2) above. But this finishes the proof in the orthogonal case, and the hyperoctahedral case follows as well, by restriction.
\end{proof}

We can now formulate our main half-liberation result here:

\begin{theorem}
Let $\widetilde{U}_N^\times\subset\widetilde{U}_N^{**}$ and $\widetilde{K}_N^\times\subset\widetilde{K}_N^{**}$ be the subspaces corresponding to the images of the maps $C(\widetilde{U}_N^{**})\to C(\widetilde{U}_{2,N})\rtimes\mathbb Z_2$ and $C(\widetilde{K}_N^{**})\to C(\widetilde{K}_{2,N})\rtimes\mathbb Z_2$ above.
\begin{enumerate}
\item $\widetilde{U}_N^\times,\widetilde{K}_N^\times$ are semigroups, with multiplication coming from the crossed products.

\item $\widetilde{U}_N\subset\widetilde{U}_N^\times\subset \widetilde{U}_N^{**}$ and $\widetilde{K}_N\subset\widetilde{K}_N^\times\subset\widetilde{K}_N^{**}$ commute with the multiplications.
\end{enumerate}
\end{theorem}

\begin{proof}
We use various formulae established above, and the method in \cite{bdu}.

(1) This is clear, from Proposition 5.4 above.

(2) This is similar to the proof of Theorem 5.6 (2) above.
\end{proof}

To summarize, several half-liberation methods are available already in the unitary group case, cf. \cite{bdd}, \cite{bdu}, and even more methods are available in the semigroup case. 

Regarding the basic semigroups $\widetilde{H}_N,\widetilde{O}_N,\widetilde{K}_N,\widetilde{U}_N$, we have constructed here several half-liberations, $\widetilde{H}_N^\times\subset\widetilde{H}_N^*$ and $\widetilde{O}_N^\times\subset\widetilde{O}_N^*$ in the orthogonal case, and $\widetilde{K}_N^\times\subset\widetilde{K}_N^{**}\subset\widetilde{K}_N^*$ and $\widetilde{U}_N^\times\subset\widetilde{U}_N^{**}\subset\widetilde{U}_N^*$ in the unitary case. All these objects seem to be quite interesting, with the ``reduced'' ones $\widetilde{H}_N^\times,\widetilde{O}_N^\times,\widetilde{K}_N^\times,\widetilde{U}_N^\times$ being compact quantum semigroups.

Finally, as pointed out in \cite{bdd}, \cite{bdu}, there are several other potentially interesting half-liberation methods, in the unitary group case. These are waiting to be better understood, and maybe even classified, and also adapted to the semigroup setting as well. This adds to the various questions raised in the introduction, and throughout the paper.

\end{document}